\documentclass{amsart}
\usepackage{amsmath,amsthm,amstext,amssymb}
\usepackage{graphicx}
\usepackage{enumitem}

\usepackage{mathtools}
\usepackage{amsmath,amssymb,amscd,amsthm,amsfonts,amstext,amsbsy,mathrsfs,upgreek,mathtools,stmaryrd,enumitem,bbm}

\usepackage{hyperref}

\theoremstyle{plain}
\newtheorem{theorem}{Theorem}[section]
\newtheorem{lemma}[theorem]{Lemma}
\newtheorem{question}[theorem]{Question}
\newtheorem{remark}[theorem]{Remark}
\newtheorem{corollary}[theorem]{Corollary}
\newtheorem*{claim*}{Claim}
\newtheorem*{subclaim*}{Subclaim}
\theoremstyle{definition}
\newtheorem{definition}[theorem]{Definition}

\newcommand{\crit}[1]{{{\rm{crit}}\left({#1}\right)}}
\newcommand{\cof}[1]{{{\rm{cof}}(#1)}}

\newcommand{\ran}[1]{{{\rm{ran}}(#1)}}

\newcommand{\length}[1]{{\rm{lh}}({#1})}
\newcommand{\POT}[1]{{\mathcal{P}}({#1})}
\newcommand{\map}[3]{{#1}:{#2}\longrightarrow{#3}}
\newcommand{\Map}[5]{{#1}:{#2}\longrightarrow{#3};~{#4}\longmapsto{#5}}

\newcommand{\Set}[2]{\{{#1}~\vert~{#2}\}}
\newcommand{\seq}[2]{\langle{#1}~\vert~{#2}\rangle}

\newcommand{\anf}[1]{{\text{``}\hspace{0.3ex}{#1}\hspace{0.3ex}\text{''}}}

\newcommand{\HH}[1]{{\rm{H}}(#1)}
\newcommand{\Ult}[2]{{\mathrm{Ult}}({#1},{#2})}
\newcommand{\Add}[2]{{\rm{Add}}({#1},{#2})}
\newcommand{\Col}[2]{{\rm{Col}}({#1},{#2})}
\newcommand{\id}{{\rm{id}}}
\newcommand{\Lim}{{\rm{Lim}}}
\newcommand{\On}{{\rm{On}}}
\newcommand{\LL}{{\rm{L}}}
\newcommand{\ZFC}{{\rm{ZFC}}}
\newcommand{\GCH}{{\rm{GCH}}}
\newcommand{\CH}{{\rm{CH}}}

\newcommand{\MMM}{{\mathbb{M}}}
\newcommand{\PPP}{{\mathbb{P}}}
\newcommand{\QQQ}{{\mathbb{Q}}}
\newcommand{\RRR}{{\mathbb{R}}}

\newcommand{\KK}{{\rm{K}}}
\newcommand{\VV}{{\rm{V}}}
\newcommand{\calT}{\mathcal{T}}
\newcommand{\calQ}{\mathcal{Q}}
\newcommand{\calM}{\mathcal{M}}

\newenvironment{enumerate-(a)}{\begin{enumerate}[label={\upshape (\alph*)}, leftmargin=2pc]}{\end{enumerate}}

\newenvironment{enumerate-(a)-r}{\begin{enumerate}[label={\upshape (\alph*)}, leftmargin=2pc,resume]}{\end{enumerate}}

\newenvironment{enumerate-(A)}{\begin{enumerate}[label={\upshape (\Alph*)}, leftmargin=2pc]}{\end{enumerate}}

\newenvironment{enumerate-(A)-r}{\begin{enumerate}[label={\upshape (\Alph*)}, leftmargin=2pc,resume]}{\end{enumerate}}

\newenvironment{enumerate-(i)}{\begin{enumerate}[label={\upshape (\roman*)}, leftmargin=2pc]}{\end{enumerate}}

\newenvironment{enumerate-(i)-r}{\begin{enumerate}[label={\upshape (\roman*)}, leftmargin=2pc,resume]}{\end{enumerate}}

\newenvironment{enumerate-(I)}{\begin{enumerate}[label={\upshape (\Roman*)}, leftmargin=2pc]}{\end{enumerate}}

\newenvironment{enumerate-(I)-r}{\begin{enumerate}[label={\upshape (\Roman*)}, leftmargin=2pc,resume]}{\end{enumerate}}

\newenvironment{enumerate-(1)}{\begin{enumerate}[label={\upshape (\arabic*)}, leftmargin=2pc]}{\end{enumerate}}

\newenvironment{enumerate-(1)-r}{\begin{enumerate}[label={\upshape (\arabic*)}, leftmargin=2pc,resume]}{\end{enumerate}}

 \title[$\Sigma_1(\kappa)$-definable subsets of H($\kappa^+$)]{$\mathbf{\Sigma_1({\boldsymbol \kappa})}$-definable subsets of H($\mathbf{\boldsymbol\kappa^+}$)}

 \author{Philipp L\"ucke}
\address{Mathematisches Institut, Universit\"at Bonn\\Endenicher Allee 60\\53115 Bonn\\Germany}
\email{pluecke@math.uni-bonn.de} 

 \author{Ralf Schindler}
 \address{Institut f\"ur Mathematische Logik und Grundlagenforschung, Universit\"at M\"unster, Einsteinstr. 62, 48149 M\"unster, Germany}  
\email{rds@math.uni-muenster.de} 

 \author{Philipp Schlicht}
  \address{Mathematisches Institut, Universit\"at Bonn\\Endenicher Allee 60\\53115 Bonn\\Germany} 
 \email{schlicht@math.uni-bonn.de}

\thanks{During the preparation of this paper, the first and the third author were partially supported by DFG-grant LU2020/1-1. This research was partially done whilst all three author were visiting fellows at the Isaac Newton Institute for Mathematical Sciences in the programme `Mathematical, Foundational and Computational Aspects of the Higher Infinite' (HIF) funded by EPSRC grant EP/K032208/1. The authors would like to thank the organizers for the opportunity to participate in this programme. Finally, the authors would also  like to thank Philip Welch for many helpful comments on an earlier version of the paper and the anonymous referee for the careful reading of the manuscript.}


\begin{document} 

\begin{abstract} 
 We study $\Sigma_1(\omega_1)$-definable sets (i.e. sets that are equal to the collection of all sets satisfying a certain $\Sigma_1$-formula with parameter $\omega_1$) in the presence of large cardinals.  
Our results show that the existence of a Woodin cardinal and a measurable cardinal above it imply that no well-ordering of the reals is $\Sigma_1(\omega_1)$-definable, the set of all stationary subsets of $\omega_1$ is not $\Sigma_1(\omega_1)$-definable and the complement of every $\Sigma_1(\omega_1)$-definable Bernstein subset of ${}^{\omega_1}\omega_1$ is not $\Sigma_1(\omega_1)$-definable.  
In contrast, we show that the existence of a Woodin cardinal is compatible with the existence of a $\Sigma_1(\omega_1)$-definable well-ordering of $\HH{\omega_2}$ and the existence of a $\Delta_1(\omega_1)$-definable Bernstein subset of  ${}^{\omega_1}\omega_1$. 
We also show that, if there are  infinitely many Woodin cardinals and a measurable cardinal above them, then there is no $\Sigma_1(\omega_1)$-definable uniformization of the club filter on $\omega_1$. 
Moreover, we prove a perfect set theorem for $\Sigma_1(\omega_1)$-definable subsets of ${}^{\omega_1}\omega_1$, assuming that there is a measurable cardinal and the non-stationary ideal on $\omega_1$ is saturated.  The proofs of these results use iterated generic ultrapowers and Woodin's $\PPP_{\mathrm{max}}$-forcing. 
Finally, we also prove variants of some of these results for $\Sigma_1(\kappa)$-definable subsets of ${}^{\kappa}\kappa$, in the case where $\kappa$ itself has certain large cardinal properties. 
\end{abstract}

\maketitle



\section{Introduction}

Given an uncountable regular cardinal $\kappa$, we study subsets of the collection $\HH{\kappa^+}$ of all sets of hereditary cardinality at most $\kappa$ that are definable over $\HH{\kappa^+}$ by simple formulas.

\begin{definition} 
 Let $M$ be a non-empty class, let $R_0,\ldots,R_{n-1}$ be relations on $M$ and let $a_0,\ldots,a_{m-1}$ be elements of $M$. Set $\MMM=\langle M,\in,R_0,\ldots,R_{n-1}\rangle$. 
 \begin{enumerate}
  \item A subset $X$ of $M$ is \emph{$\Sigma_1(a_0,\ldots,a_{m-1})$-definable over $\MMM$} if there is a $\Sigma_1$-formula $\varphi(v_0,\ldots,v_m)$ in the language of set theory extended by predicate symbols $\dot{P}_0,\ldots,\dot{P}_{n-1}$ such that $X  = \Set{x\in M}{\MMM\models\varphi(a_0,\ldots,a_{m-1},x)}$.   
  
  \item A subset $Y$ of $M$ is \emph{$\Pi_1(a_0,\ldots,a_{m-1})$-definable over $\MMM$} if $M\setminus Y$ in $M$ is $\Sigma_1(a_0,\ldots,a_{m-1})$-definable over $\MMM$.
  
  \item A subset of $M$ is \emph{$\Delta_1(a_0,\ldots,a_{m-1})$-definable over $\MMM$} if the subset  is both $\Sigma_1(a_0,\ldots,a_{m-1})$- and $\Pi_1(a_0,\ldots,a_{m-1})$-definable over $\MMM$. 
 \end{enumerate} 
\end{definition}

Since $\Sigma_1$-formulas 
are absolute between $\VV$ and $\HH{\kappa^+}$, we will not mention the models $\langle\VV,\in\rangle$ and $\langle\HH{\kappa^+},\in\rangle$ in our statements about $\Sigma_1$-definability.

In this paper, we will focus on the following subjects: $\Sigma_1(\kappa)$-definable well-orderings of $\HH{\kappa^+}$, $\Delta_1(\kappa)$-definitions of the club filter on $\kappa$ and $\Delta_1(\kappa)$-definable Bernstein subsets of ${}^\kappa\kappa$ (see Definition \ref{definition:PerfectBernstein} below). In the case of formulas containing arbitrary parameters from $\HH{\kappa^+}$, it  was shown that the existence of such objects is independent from $\ZFC$ together with large cardinal axioms (see \cite{HL}, \cite{MR2987148} and \cite{MR1242054}). Moreover, it is known that such $\Sigma_1(\kappa)$-definitions exists in certain models of set theory that do not contain larger large cardinals (see \cite{MR3320477} and \cite{MR3591274}). This leaves open the question whether such $\Sigma_1(\kappa)$-definitions are compatible with larger large cardinals. 
The main results of this paper show that large cardinal axioms imply the non-existence of such definitions for $\kappa=\omega_1$.

Using results of Woodin on the $\Pi_2$-maximality of the $\PPP_{max}$-extension of $\LL(\RRR)$ (see \cite{MR2768703} and \cite{MR1713438}), it is easy to show that the assumptions that there are infinitely many Woodin cardinals with a measurable cardinal above them all implies that no well-ordering of the reals is $\Sigma_1(\omega_1)$-definable.   We will derive this conclusion from a much weaker assumption that is  in some sense optimal (see remarks below).

\begin{theorem}\label{theorem:Omega1WO}
 Assume that there is a Woodin cardinal and a measurable cardinal above it. Then no well-ordering of the reals is $\Sigma_1(\omega_1)$-definable. 
\end{theorem}

In contrast, we will show that the existence of a $\Sigma_1(\omega_1)$-definable well-ordering of $\HH{\omega_2}$ is compatible with the existence of a Woodin cardinal (see Theorem \ref{Sigma1 wellorder in M1}). Together with the above theorem, this answers {\cite[Question 1.9]{MR3591274}}.

Given a regular cardinal  $\kappa$, the \emph{generalized Baire space} for $\kappa$ consists of the  set ${}^\kappa\kappa$ of all functions from $\kappa$ to $\kappa$ equipped with the topology whose basic open sets are of the form $N_s=\Set{x\in{}^\kappa\kappa}{s\subseteq x}$ for some $\map{s}{\alpha}{\kappa}$ with $\alpha<\kappa$.

\begin{definition}\label{definition:PerfectBernstein}
 Let $\kappa$ be a regular cardinal.
 \begin{enumerate} 
  \item 
  A \emph{perfect subset} of ${}^{\kappa}\kappa$ is the set of branches $[T]$ of a \emph{perfect subtree} of ${}^{{<}\kappa}\kappa$, i.e. a ${<}\kappa$-closed tree with branching nodes above all nodes. 

  \item A subset $A$ of ${}^{\kappa}\kappa$ has the \emph{perfect set property} if either $A$ has cardinality at most $\kappa$ or $A$ contains a perfect subset. 
  

  \item A \emph{Bernstein set} is a subset of ${}^{\kappa}\kappa$ with the property that neither $A$ nor its complement contains a perfect subset.  
 \end{enumerate}
\end{definition}

\begin{theorem}\label{theorem:Omega1Bernstein}
 Assume that there is a Woodin cardinal and a measurable cardinal above it. Then no Bernstein subset of ${}^{\omega_1}\omega_1$ is $\Delta_1(\omega_1)$-definable over $\langle\HH{\omega_2},\in\rangle$. 
\end{theorem}

We will also show that the large cardinal assumption of the above result is close to optimal by showing that the existence of such a Bernstein subset is compatible with the existence of a Woodin cardinal (see Lemma \ref{Delta1 Bernstein set in M_1}).

Next, we consider $\Delta_1(\omega_1)$-definitions of the club filter $\mathrm{C}_{\omega_1}$ and the nonstationary ideal $\mathrm{NS}_{\omega_1}$ on $\omega_1$. In \cite{FriedmanWu}, Friedman and Wu showed that the existence of a proper class of Woodin cardinals implies that $\mathrm{NS}_{\omega_1}$ is not $\Delta_1(\omega_1)$-definable. We will derive a stronger conclusion from a weaker hypothesis. 
 In the following, we say that a subset $X$ of $\POT{\kappa}$ \emph{separates the club filter from the nonstationary ideal } if $X$ contains $\mathrm{C}_{\omega_1}$ as a subset and is disjoint from $\mathrm{NS}_{\omega_1}$.

\begin{theorem}\label{theorem:Omega1Club}
 Assume that there is a Woodin cardinal and a measurable cardinal above it. Then no subset of $\POT{\omega_1}$ that separates the club filter from the nonstationary ideal is $\Delta_1(\omega_1)$-definable over $\langle\HH{\omega_2},\in\rangle$.  
\end{theorem}

 We will in fact prove more general versions of the above theorems. First, we will derive the above conclusions from the assumption that $M_1^\#(A)$ exists for every subset $A$ of $\omega_1$ (see {\cite[p. 1738]{MR2768699}} and {\cite[p. 1660]{MR2768698}}). This assumption follows from the existence of a Woodin cardinal and a measurable cardinal above it (see \cite{MR1300637} and \cite{MR1257469}). 
 In Section \ref{section: forcing axioms}, we will show that it also follows from $\mathsf{BMM}$ (Bounded Martin's Maximum) together with the assumption that the nonstationary ideal $\mathrm{NS}_{\omega_1}$ on $\omega_1$ is precipitous.  Second, we will allow as parameters subsets of $\omega_1$ that are $\mathbf{\Sigma}^1_2$-definable in the codes. We will also prove this for all subsets of $\omega_1$ which are universally Baire in the codes, assuming that there is a proper class of Woodin cardinals.  
Finally, we will prove results on 
perfect subsets of $\Sigma_1(\omega_1)$-subsets of ${}^{\omega_1}\omega_1$ (see Section \ref{subsection:PerfectSet}),  
the nonexistence of $\Sigma_1(\omega_1)$-definable uniformizations of the club filter (see Section \ref{subsection:Uniformizations of the club filter}) and the absoluteness of $\Sigma_1(\omega_1)$-statements (see Section \ref{subsection:absoluteness}).

The above results raise the question whether large cardinals have a similar influence on $\Sigma_1(\kappa)$-definability for regular cardinals $\kappa>\omega_1$. Variations of the techniques used in the proofs of the above results will allow us to prove 
analogous statements hold for $\Sigma_1(\kappa)$-definable subsets of $\HH{\kappa^+}$ in the case where $\kappa$ itself has certain large cardinal properties.

\begin{theorem}\label{theorem:LargeCardinalsWO}
 If $\kappa$ is either a measurable cardinal above a Woodin cardinal or a Woodin cardinal below a measurable cardinal, then there is no $\Sigma_1(\kappa)$-definable well-ordering of the reals. 
\end{theorem}

\begin{theorem}\label{theorem:LargeCardinalsBernstein}
 If  $\kappa$ is a measurable cardinal with the property that there are two distinct  normal ultrafilters on $\kappa$, then no Bernstein subset of ${}^\kappa\kappa$ is $\Delta_1(\kappa)$-definable over $\langle\HH{\kappa^+},\in\rangle$.  
\end{theorem}

In contrast, we will show that consistently there can be a measurable cardinal $\kappa$ and a Bernstein subset of ${}^\kappa\kappa$ that is $\Delta_1(\kappa)$-definable over $\langle\HH{\kappa^+},\in\rangle$.

Next, we consider the $\Delta_1(\kappa)$-definability of sets separating the club filter from the non-stationary ideal at \emph{$\omega_1$-iterable cardinals} (see Definition \ref{definition:IterableCardinal}).

\begin{theorem}\label{theorem:MeasurableCardinalClubFilter}
 If $\kappa$ is an $\omega_1$-iterable cardinal and $X$ is a subset of $\POT{\kappa}$ that separates the club filter from the nonstationary ideal, then $X$ is not $\Delta_1(\kappa)$-definable over $\langle\HH{\kappa^+},\in\rangle$. 
\end{theorem}

Friedman and Wu showed that the club filter on $\kappa$ is not $\Pi_1(\kappa)$-definable over $\langle\HH{\kappa^+},\in\rangle$ if $\kappa$ is a weakly compact cardinal (see {\cite[Proposition 2.1]{FriedmanWu}}). We will show that this conclusion also holds for stationary limits of $\omega_1$-iterable cardinals. Note that these cardinal need not be weakly compact and Woodin cardinals are stationary limits of $\omega_1$-iterable cardinals. Moreover, {\cite[Lemma 5.2]{MR2817562}} shows that $\omega_1$-Erd\H{o}s cardinals are stationary limits of $\omega_1$-iterable cardinals.

\begin{theorem}\label{theorem:IterableCardinalsClubFilter}
  If $\kappa$ is a regular cardinal that is a stationary limit of $\omega_1$-iterable cardinals, then the club filter on $\kappa$ is not $\Pi_1(\kappa)$-definable over $\langle\HH{\kappa^+},\in\rangle$. 
\end{theorem}

We outline the content of this paper. In Section \ref{section: forcing axioms} we will show that the condition that $M_1^{\#}(A)$ exists for all subsets $A$ of $\omega_1$ follows from $\mathsf{BMM}$ and the assumption that the non-stationary ideal $\mathrm{NS}_{\omega_1}$ on $\omega_1$ is precipitous. 
In Section \ref{section: Sigma-1 and Sigma-1-3} we characterize $\Sigma_1(\omega_1)$-definable sets of reals and extend this characterization to formulas with universally Baire parameters, assuming that there is a proper class of Woodin cardinals. 
In Section \ref{section: main results} we prove the main results about $\Sigma_1(\omega_1)$-definable subsets of $H(\kappa^+)$. 
In Section \ref{section: M1} we show that the assumptions of some of the previous results are optimal by showing that some of the results fail in $M_1$. 
In Section \ref{section: large cardinals} we prove version of some of the previous results for $\Sigma_1(\kappa)$-definable subsets of $H(\kappa^+)$, where $\kappa$ is a large cardinal, for instance a measurable cardinal or an $\omega_1$-iterable cardinal. We close this paper by listing several open questions motivated by the above results in Section \ref{section:Questions}.


\section{Forcing axioms and $M_1^{\#}(A)$} \label{section: forcing axioms} 

We will frequently make use of the hypothesis that $M_1^{\#}(A)$ exists for every subset $A$ of $\omega_1$. 
We show that this follows from $\mathsf{BMM}$ together with the assumption that the nonstationary ideal $\mathrm{NS}_{\omega_1}$ on $\omega_1$ is precipitous, by varying arguments from \cite{MR2963017}. 

\begin{theorem}
Assume $\mathsf{BMM}$ and that $\mathrm{NS}_{\omega_1}$ is precipitous.
Then $M_1^{\#}(A)$ exists for every $A\subseteq\omega_1$. 
\end{theorem}

\begin{proof} 
Let us first assume that there is no inner model with a Woodin cardinal,
and let $\KK$ denote the core model (see for example \cite{MR3135495}).
By {\cite[Theorem 0.3]{MR2963017}}, 
the fact that $\mathrm{NS}_{\omega_1}$ is precipitous (or just the fact that there is a 
normal precipitous ideal on $\omega_1$) yields $(\omega_1^\VV)^{+\KK}
= \omega_2^\VV$, whereas by
{\cite[Lemma 7.1]{MR2963017}}, $\mathsf{BMM}$ (or just $\mathsf{BPFA}$)
gives that $(\omega_1^\VV)^{+\KK}
< \omega_2^\VV$. This is a plain contradiction, so that there must be an inner 
model with
a Woodin cardinal.

By {\cite[Theorem 1.3]{MR2096166}}, $\mathsf{BMM}$ yields that $\VV$ is closed
under $X \mapsto X^\#$.
By a theorem of Woodin, the facts that there is an inner model with a 
Woodin cardinal and $\VV$ is closed under the sharp operation imply 
that $M_1^\#$ exists and is fully iterable.\footnote{This result is unpublished, but the methods used in the (known) proof can be found in \cite{CabalVolume3}.}  
This argument relativizes to show that for any real $x$, $M_1^\#(x)$
exists and is fully iterable.

Let us now fix $A \subseteq \omega_1$ and prove that $M_1^\#(A)$ exists and is
countably iterable.
Let $\map{j}{\VV}{M \subseteq \VV[G]}$, 
where $G$ is $\mathrm{NS}_{\omega_1}$-generic over $\VV$ and
$j$ is the induced generic elementary embedding such that $M$ is transitive.
By elementarity, $M_1^\#(A)$ exists in $M$ and is fully iterable in $M$.
We aim to see that $(M_1^\#(A))^M \in\VV$ and it is fully iterable in $\VV$.

As $\VV$ is closed under the sharp operation, $F = \Set{ \langle x, x^\# \rangle}{x\in\RRR}$ is universally Baire. Suppose that $T$ and $U$ are (class sized) trees such that
$F=p[T]$ in $\VV$ and $p[U] = {\mathbb R}^2 \setminus
p[T]$ in every generic extension of $\VV$. By well-known arguments, we must have $p[j(T)] = p[T]$ in $\VV[G]$ and in fact in every generic extension of $\VV[G]$. 

We first claim that $(M_1^\#(A))^M$ is $\omega_1$-iterable in $\VV[G]$ and in fact
in every generic extension $\VV[G,H]$ of $\VV[G]$ via its unique iteration strategy. 
In order to see this, let $W \in M$ be a canonical tree of attempts to find

\begin{enumerate}
\item[(a)] $\map{\sigma}{N}{(M_1^\#(A))^M}$, where $N$ is countable,
\item[(b)] $\calT$ is a countable iteration tree on $N$ 
\item[(c)] $\seq{\calQ_\lambda}{\lambda \in {\rm Lim} \cap ({\rm lh}(\calT)+1)}$ is such that for
every $\lambda \in\Lim \cap ({\rm lh}(\calT)+1)$, $\calQ_\lambda \trianglelefteq
(\mathcal{M}(\calT \upharpoonright \lambda))^\#$ is a $\calQ$-structure for
$\mathcal{M}(\calT)$, and for every ordinal $\lambda \in\Lim \cap {\rm lh}(\calT)$, we have $\calQ_\lambda \trianglelefteq \calM_\lambda^\calT$, and either
 \begin{enumerate}
  \item[(d1)] $\calT$ has a last ill-founded model, or else 

  \item[(d2)] $\calT$ has limit length but no cofinal branch $b$ such that
$\calQ_{{\rm lh}(\calT)} \trianglelefteq \calM_b^\calT$.
 \end{enumerate}
\end{enumerate}

Notice that we may use $j(T)$ to certify the first part of (c). 
If $(M_1^\#(A))^M$ were not $\omega_1$-iterable in $\VV[G,H]$, then
$W$ would be ill-founded in $\VV[G,H]$, hence in $M$, and then
$(M_1^\#(A))^M$ would not be iterable in $M$. Contradiction!

Let $\map{j'}{\VV}{M' \subseteq\VV[H] \subseteq \VV[G,H]}$, 
where $H$ is $(\mathrm{NS}_{\omega_1})^\VV$-generic over $\VV[G]$ and
$j'$ is the induced generic elementary embedding such that $M'$ is transitive.
By the above argument, $(M_1^\#(A))^M$ and $(M_1^\#(A))^{M'}$
may be successfully coiterated inside $\VV[G,H]$, so that in fact
$(M_1^\#(A))^M = (M_1^\#(A))^{M'}$, and hence
$(M_1^\#(A))^M \in\VV$.

Assume $(M_1^\#(A))^M$ were not $\omega_1$-iterable in some generic
extension $\VV[H]$ of $\VV$. We may without loss of generality assume that
$H$ is generic over $\VV[G]$. Let $W' \in\VV$ be defined exactly 
as the tree $W$ above, except for that
we use $T$ instead of $j(T)$ to certify the first part of (c). 
By $p[j(T)]=p[T]$ in $\VV[G,H]$, we must have $p[W'] = p[W]$ in $\VV[G,H]$.
As we assume $(M_1^\#(A))^M$ to be not $\omega_1$-iterable in 
$\VV[G,H]$, $W'$ would be
ill-founded in $\VV[G,H]$, so that $W$ would be ill-founded in $\VV[G,H]$ and
hence in $M$. Contradiction!

The argument given shows that $(M_1^\#(A))^M \in\VV$ is fully 
iterable in $\VV$.
\end{proof}


\section{$\Sigma_1(\omega_1)$-definable sets and $\Sigma^1_3$ sets} \label{section: Sigma-1 and Sigma-1-3}

We give a characterization of $\Sigma_1(\omega_1)$-definable sets of reals which we will use in the proof of Theorem \ref{theorem:Omega1WO}. Let $\mathrm{WO}$ denote the $\Pi^1_1$-set of all reals that code a well-ordering of $\omega$ (in some fixed canonical way) and, given $z\in\mathrm{WO}$, let $\Vert z\Vert$ denote the order-type of the well-ordering coded by $z$.  Remember that, given a class $\Gamma$ of subsets of $\RRR$, a subset $A$ of $\omega_1$ is \emph{$\Gamma$ in the codes} if there is $W\in\Gamma$ such that $A=\Set{\Vert z\Vert}{z\in W\cap\mathrm{WO}}$.  Note that $\omega_1$ is $\Sigma^1_2$ in the codes.

\begin{lemma}\label{Sigma13HigherCardinals}
 If $a\in\RRR$, $X$ is a $\Sigma^1_3(a)$-subset of $\RRR$ and $\kappa$ is an uncountable cardinal, then $X$ is $\Sigma_1(\kappa,a)$-definable. 
\end{lemma}

\begin{proof}
Pick a $\Sigma^1_3$-formula $\psi(v_0,v_1)$ that defines $X$ using the parameter $a$. In this situation, Shoenfield absoluteness implies that the set $X$ is equal to the set of all $x\in\RRR$ with the property that there is a transitive model $M$ of $\ZFC^-$ in $\HH{\kappa^+}$ such that $a,x\in M$, $\kappa\subseteq M$ and $\psi(a,x)^M$. This  yields a $\Sigma_1(\kappa,a)$-definition of $X$.   
\end{proof}

 In the following, we will show that the converse of the above implication for $\omega_1$ holds in the presence of large cardinals. This argument makes use of the \emph{countable stationary tower $\QQQ_{{<}\delta}$} introduced by Woodin (see {\cite[Section 2.7]{MR2069032}}) and results of Woodin on generic iteration (see {\cite[Lemma 3.10 \& Remark 3.11]{MR1713438}).

\begin{lemma} \label{omega1 iterable structure} 
 Let $M$ be a transitive model of $\mathsf{ZFC}^-$ with a largest cardinal $\kappa$ and let $\PPP$ be a partial order in $M$ of cardinality less than $\kappa$ such that the following conditions hold: 
\begin{enumerate} 
 \item Forcing with $\mathbb{P}$ adds a $(\mu,\nu)$-extender over $M$ for some $\mu,\nu<\kappa$.  

 \item There is an $\omega_1$-iterable $M$-ultrafilter $U$ on $\kappa$.  
\end{enumerate} 
Then $M$ is $\omega_1$-iterable with respect to $\mathbb{P}$ and its images. 
\end{lemma} 

\begin{proof} 
 We first suppose that $M$ is countable.  Let $$\langle\seq{M^\alpha_0}{\alpha<\omega_1},  ~ \seq{\map{j^{\alpha,\beta}_0}{M^\alpha_0}{M^\beta_0}}{\alpha\leq\beta<\omega_1} \rangle$$ denote the iteration of $\langle M,\in,U\rangle$ of length $\omega_1$. Set $\kappa_0^\alpha=j_0^{0,\alpha}(\kappa)$ for all $\alpha<\omega_1$. 
Given $\alpha<\beta<\omega_1$, we have $M^\alpha_0=\HH{(\kappa^\alpha_0)^+}^{M^\alpha_0}=\HH{(\kappa^\alpha_0)^+}^{M^\beta_0}$.

In the following, we show that $M=M^0_0$ is $\alpha$-iterable for all $\alpha<\omega_1$.  
Assume that $\alpha$ is the least counterexample to this statement. Then $\alpha$ is a limit ordinal by our assumptions on $\PPP$. 
Suppose that $\seq{G_\beta}{\beta<\alpha}$ is a sequence of generic filters given by a generic iteration of $M_0^0$ by the extenders added by $\PPP$ its images such that the corresponding direct limit is ill-founded. 
Let $$\langle\seq{M_\beta^0}{\beta<\alpha}, ~ \seq{\map{k^0_{\beta,\gamma}}{M^0_\beta}{ M^0_\gamma}}{\beta\leq\gamma<\alpha}$$ denote the corresponding system of models and elementary embeddings.

By simultaneous recursion, we define a system 
\begin{itemize} 
 \item $\seq{M^\beta_\gamma}{\beta\leq\alpha+1, ~ \gamma<\alpha}$ of transitive models of $\mathsf{ZFC}^-$, and 

 \item commuting systems 
\[\seq{\map{j_\delta^{\beta,\gamma}}{M^\beta_\delta}{M_\delta^\gamma}}{\beta\leq\gamma\leq\alpha+1, ~ \delta<\alpha}\]  
\[\seq{\map{k^\beta_{\gamma,\delta}}{M^\beta_\gamma}{M^\beta_\delta}}{\beta\leq\alpha+1, \gamma\leq\delta<\alpha}\]  
of elementary embeddings. 
\end{itemize} 
such that the following properties can be verified by simultaneous induction,where $\kappa^\gamma_\delta=k^\gamma_{0,\delta}(\kappa^\gamma_0)$ for $\gamma<\alpha+1$ and $\delta<\alpha$: 
\begin{itemize} 
 \item $M^\beta_\delta=\HH{(\kappa^\beta_\delta)^+}^{M^\gamma_\delta}$ for all  $\beta\leq\gamma\leq\alpha+1$ and $\delta<\alpha$. 

 \item $j_\delta^{\beta,\gamma}\restriction\HH{\kappa^\beta_\delta}^{M^\beta_\delta}$ is the identity on $\HH{\kappa^\beta_\delta}^{M^\beta_\delta}$. 

 \item For all $\beta\leq\alpha+1$, $\seq{M^\beta_\gamma}{\gamma<\alpha}$ is the generic iteration of $M^\beta_0$ by $\PPP=j_0^{0,\beta}(\PPP)$ and its images $\seq{k^0_{0,\beta}(\PPP)}{\beta<\alpha}$ using the sequence $\seq{G_\gamma}{\gamma<\alpha}$ of filters. 
\end{itemize} 

By {\cite[Lemma 3.10 \& Remark 3.11]{MR1713438}}, the sequence $\seq{M^{\alpha+1}_\gamma}{\gamma<\alpha}$ has a well-founded direct limit. Since the system of elementary embeddings commutes, the direct limit of the sequence $\seq{M^0_\gamma}{\gamma<\alpha}$ embeds into this model and hence it is well-founded, a contradiction. 

For arbitrary models $M$, the claim follows by forming a countable elementary substructure of some $\HH{\theta}$.  
\end{proof}

\begin{lemma}\label{lemma:HOmega2Sigma13}
 Assume that $M_1^\#(A)$ exists for every $A\subseteq\omega_1$. Given $a\in\RRR$, the following conditions are equivalent for any subset  $X$ of $\RRR$. 
 \begin{enumerate} 
 \item  $X$ is $\Sigma_1(A)$-definable for some  $A\subseteq\omega_1$ that is $\Sigma^1_2(a)$ in the codes. 

 \item  $X$ is a $\Sigma^1_3(a)$-subset of $\RRR$. 
 \end{enumerate} 
\end{lemma}

\begin{proof}
  Assume that (i) holds. Fix a $\Sigma_1$-formula $\varphi(v_0,v_1)$ and a $\Sigma^1_2$-formula $\psi(v_0,v_1)$ with the property that  $X=\Set{x\in\RRR}{\varphi(A,x)}$, where $A=\Set{\Vert z\Vert}{z\in W\cap\mathrm{WO}}$ and $W=\Set{z\in\RRR}{\psi(a,z)}$.  
  Define $Y$ to be the set of all $y\in\RRR$ with the property that there is a countable transitive model $M$ of $\ZFC^-$ and $\delta,A_0,W_0\in M$ such that $a,y\in M$ and the following statements hold: 
 \begin{enumerate}
  \item $\delta$ is a Woodin cardinal in $M$ and $M$ is $\omega_1$-iterable with respect to $\QQQ_{{<}\delta}^M$ and its images.   

  \item In $M$, we have $W_0=\Set{z\in\RRR}{\psi(a,z)}$, $A_0=\Set{\Vert z\Vert}{z\in W_0\cap\mathrm{WO}}$ and $\varphi(A_0,y)$ holds. 
 \end{enumerate}

 \begin{claim*}
  The set $Y$ is a $\Sigma^1_3(a)$-subset of $\RRR$. 
 \end{claim*}
 
 \begin{proof} 
 The only condition on $M$ which is not first-order is $\omega_1$-iterability. 
 This condition states that all countable generic iterates are well-founded and hence it is a $\Pi^1_2$-statement. 
  \end{proof}

 \begin{claim*}
  $Y\subseteq X$. 
 \end{claim*}

 \begin{proof}
  Fix $y\in Y$  and pick  a countable transitive model $M_0$ and $\delta,A_0,W_0\in M_0$ witnessing this.   Let $\seq{M_\alpha}{\alpha\leq\omega_1}$ be a generic iteration of $M_0$ using $\QQQ_{{<}\delta}^{M_0}$ and its images. Set $N=M_{\omega_1}$ and let $\map{j}{M_0}{N}$ denote the corresponding elementary embedding. Then $N$ is a transitive model of $\ZFC^-$ and $j(\omega_1^{M_0})=\omega_1^N=\omega_1$. Pick $\alpha\in A$. Then there is $u\in\mathrm{WO}^ N$ such that $\alpha=\alpha_u$ and $$\exists z\in\mathrm{WO} ~ [\Vert u\Vert=\Vert z\Vert ~ \wedge ~ \psi(a,z)]$$ holds. Since $\omega_1\subseteq N$, Shoenfield absoluteness implies that there is $z\in\mathrm{WO}^N$ with $\alpha=\Vert z\Vert$ and $\psi(a,z)^N$. By elementarity, this shows that $z\in j(W_0)$ and $\alpha\in j(A_0)$. In the other direction, fix $z\in j(W_0)$. Then $\psi(a,z)^N$ holds and Shoenfield absoluteness implies that $z\in W$ and $\Vert z\Vert\in A$. We can conclude that $A=j(A_0)$ and  $\varphi(A,y)^N$ holds. By $\Sigma_1$-upwards absoluteness, this shows that $\varphi(A,y)$ holds and hence $y\in X$.  
 \end{proof}

 \begin{claim*}
  $X\subseteq Y$. 
 \end{claim*}

 \begin{proof}
  Pick $x\in X$. Then $\varphi(A,x)$ holds and we can find a subset  $C$ of $\omega_1$ such that $a,x,A\in M_1^{\#}(C)$,  $\omega_1=\omega_1^{M_1^{\#}(C)}$ and $\varphi(A,x)^{M_1^{\#}(C)}$ holds. Shoenfield absoluteness implies that $$\bar{W} ~ = ~ W\cap M_1^\#(C) ~ = ~ \Set{z\in\mathrm{\RRR}^{M_1^\#(C)}}{\psi(a,z)^{M_1^\#(C)}} ~  \in ~ M_1^\#(C).$$  As in the proof of the above claim, we can now use Shoenfield absoluteness to see that $A=\Set{\Vert z\Vert}{z\in\bar{W}\cap\mathrm{WO}^{M_1^\#(C)}}$.

 Let $N$ be a countable elementary submodel of $M_1^\#(C)$  
 and let $\map{\pi}{N}{M}$ denote the corresponding transitive collapse. 
 Then $M$ is a countable transitive model of $\ZFC^-$ with $a,x\in M$ and there is $\delta\in M$ such that $\delta$ is a Woodin cardinal in $M$ and $M$ is iterable with respect $\QQQ_{{<}\delta}^M$ and its images by Lemma \ref{omega1 iterable structure}.  
In $M$, we have  $\pi(\bar{W})=\Set{z\in\RRR}{\psi(a,z)}$, $\pi(A)=\Set{\Vert z\Vert}{z\in \pi(\bar{W})\cap\mathrm{WO}}$ and $\varphi(\pi(A),x)$ holds.    Together, this shows that $M$ and $\delta,\pi(A),\pi(\bar{W})\in M$ witness that $x$ is an element of $Y$. 
 \end{proof}

 This completes the proof of the implication from (i) to (ii). The converse implication is a direct consequence of Lemma \ref{Sigma13HigherCardinals}. 
\end{proof}

 Note that the assumptions of Lemma \ref{lemma:HOmega2Sigma13} hold for instance in $M_2$.

\begin{remark} 
The assumptions for the implication from (i) to (ii) in Lemma \ref{lemma:HOmega2Sigma13} are optimal in the following sense:  
\begin{enumerate} 
\item 
The implication is not a theorem of $\ZFC$.  If $\CH$ holds and the set $\{\RRR\}$ is $\Sigma_1(\omega_1)$-definable, then the projective truth predicate is a $\Sigma_1(\omega_1)$-definable subset of $\RRR$ that is not projective. Note that the above assumptions holds for instance in $\LL$. Moreover, we will later prove results that show that the assumption also holds  in $M_1$ (see Lemma \ref{Sigma1 wellorder in M1}). This shows that the implication does not follow from the existence of a single Woodin cardinal.  

\item The implication does not follow from $\neg\mathrm{CH}$. 
Suppose that $\LL[G]$ is an $\Add{\omega}{\omega_2}$-generic extension of $\LL$.  Since $\{\RRR^\LL\}$ is $\Sigma_1(\omega_1)$-definable in $\LL[G]$, 
the projective truth predicate of $\LL$ is $\Sigma_1(\omega_1)$-definable in $\LL[G]$. Assume that this set is projective in $\LL[G]$. By a result of Woodin (see {\cite[Lemma 9.1]{MR2987148}}), there is an $\Add{\omega}{\omega_1}$-generic filter $H$ over $\LL$ and an elementary embedding of $\LL(\RRR)^{\LL[H]}$ into $\LL(\RRR)^{\LL[G]}$.  
Then the projective truth predicate of $\LL$ is also projective in $\LL[H]$. Since $\Add{\omega}{\omega_1}$ is definable over $\HH{\omega_1}^\LL$ and satisfies the countable chain condition, the forcing relation for $\Add{\omega}{\omega_1}$ for projective statements with parameters in $\RRR^\LL$ is projective in $\LL$. Using the homogeneity of $\Add{\omega}{\omega_1}$, this shows that the projective truth predicate is projective in $\LL$, a contradiction.  
\end{enumerate} 
\end{remark}

A simpler version of the proof of Lemma \ref{lemma:HOmega2Sigma13}, using Lemma \ref{omega1 iterable structure} and generic iterations of countable substructures of $\HH{\theta}$, where $\theta$ is above a measurable cardinal, yields the following result.

\begin{lemma} \label{characterization of Sigma1(omega1) from precipitous ideal} 
The equivalence in Lemma \ref{lemma:HOmega2Sigma13} holds if there is a precipitous ideal on $\omega_1$ and a measurable cardinal. \qed 
\end{lemma}

In the following, we will add a predicate $A$ for sets of reals to the language to obtain a stronger version of Lemma \ref{lemma:HOmega2Sigma13}. 
Note that quantifiers over $A$ are unbounded in this language.  
We consider universally Baire (uB) subsets of $\RRR$. 

\begin{definition}\label{definition:Biterable}
Suppose that $\langle M,\in,I\rangle$ is a countable transitive model of $\mathsf{ZF}^-$ and $B\subseteq\RRR$. 
The structure $\langle M,\in,I\rangle$ is \emph{$B$-iterable} if the following conditions hold. 
\begin{enumerate} 
 \item $\langle M,\in,I\rangle$ is $\omega_1$-iterable, i.e. all countable iterates are well-founded. 
 
 \item $B\cap M\in M$. 

 \item If $\map{i}{M}{N}$ is a countable iteration, then $i(B\cap M)=B\cap N$. 
\end{enumerate} 
\end{definition}

Suppose that $B$ is a subset of $\RRR$. A set of reals is $\Sigma^1_n(B)$  if it is defined by a $\Sigma^1_n$-formula, where $x\in B$ and $x\notin B$ are allowed as atomic formulas.

\begin{lemma} \label{characterization of Sigma1 with uB parameter} 
 Assume that there is a proper class of Woodin cardinals. If $B$ is a uB set of reals and $X$ is a subset of $\RRR$ that  is $\Sigma_1(\omega_1)$-definable over $\langle\HH{\omega_2},\in,B,\mathrm{NS}_{\omega_1}\rangle$, then $X$ is a $\Sigma^1_3(B)$-subset of $\RRR$. 
\end{lemma} 

\begin{proof} 
Suppose that $X$ is defined by a $\Sigma_1$-formula $\varphi(v_0,v_1)$ over the structure  $\langle\HH{\omega_2},\in,B,\mathrm{NS}_{\omega_1}\rangle$.  We define $Y$ as the set of all reals $x$ such that there is a $B$-iterable structure $\langle M,\in,I\rangle$ with the property that $M$ is a model of $\ZFC$, $x\in M$ and $\langle M,\in,B\cap M,I\rangle\models \varphi(\omega_1^M,x)$.

\begin{claim*}
 $Y\subseteq X$. 
\end{claim*}

\begin{proof}
Suppose that $x\in Y$ and that this is witnessed by a $B$-iterable structure $\langle M,\in,I\rangle$. In this situation, the proof of {\cite[Lemma 4.36]{MR1713438}} shows that there is an iteration $\map{j}{\langle M,\in,I\rangle}{\langle M',\in,I'\rangle}$ of length $\omega_1$ with $I'=\mathrm{NS}_{\omega_1}\cap M'$. Since $M$ is $B$-iterable, we also have $j(B\cap M)=B\cap M'$. 
This shows that $\varphi(\omega_1,x)$ holds in $\langle M',\in,B\cap M', \mathrm{NS}_{\omega_1}\cap M'\rangle$ and therefore we can conclude that the statement also holds in $\langle\HH{\omega_2},\in,B,\mathrm{NS}_{\omega_1}\rangle$.  
\end{proof}

\begin{claim*}
 $X\subseteq Y$. 
\end{claim*}

\begin{proof}
 Suppose that $x\in X$. 
We first argue that the required $B$-iterable structure exists in a generic extension. 
 By our large cardinal assumptions and {\cite[Theorem 3]{MR560220}}, there is a generic extension $\VV[G]$ of $\VV$ with the property that $\mathrm{NS}_{\omega_1}$ is precipitous in $\VV[G]$. 
Suppose that $\mu$ is the least measurable cardinal in $\VV[G]$ 
and $\nu$ is the least inaccessible cardinal above $\mu$ in $\VV[G]$. 
Let $T$ and $U$ be trees in $\VV$ with $p[T]=B$ and $p[U]=\RRR\setminus B$ witnessing that $B$ is uB. Define $M=\VV[G]_{\nu}$, $I=\mathrm{NS}_{\omega_1}^{\VV[G]}$ and $B_G=p[T]^{\VV[G]}$.

\begin{subclaim*}
 $\langle M,\in,I,B_G\rangle\models\varphi(\omega_1^{\VV[G]},x)$. 
\end{subclaim*}

\begin{proof}
 Assume, towards a contradiction, that $\varphi(\omega_1^{\VV[G]},x)$ does not hold in the structure $\langle M,\in,I,B_G\rangle$. Let $F$ be $\PPP_{max}$-generic over $\LL(B,\RRR)^\VV$. Since our assumption implies that $$\langle\HH{\omega_2}^{\VV[G]},\in,I,B_G\rangle\models\forall\alpha ~ [\anf{\alpha=\omega_1}\longrightarrow\neg\varphi(\alpha,x)].$$ holds and this statement can be expressed by a $\Pi_2$-formula with parameter $x$, our large cardinal assumptions allow us to use results of Woodin on the $\Pi_2$-maximality of $\PPP_{max}$-extensions of $\LL(B,\RRR)$ (see   {\cite[Theorem 1.1]{MR2374762}}) to conclude that $$\langle\HH{\omega_2}^{\LL(B,\RRR)[F]},\in,\mathrm{NS}_{\omega_1}^{\LL(B,\RRR)[F]},B\rangle\models\neg\varphi(\omega_1^\VV,x)$$ holds. But we also know that $$\langle\HH{\omega_2}^\VV,\in,\mathrm{NS}_{\omega_1}^\VV,B\rangle\models\forall\alpha ~ [\anf{\alpha=\omega_1}\longrightarrow\varphi(\alpha,x)]$$ holds and hence the same theorem implies that $$\langle\HH{\omega_2}^{\LL(B,\RRR)[F]},\in,\mathrm{NS}_{\omega_1}^{\LL(B,\RRR)[F]},B\rangle\models\varphi(\omega_1^\VV,x)$$ holds, a contradiction.  
\end{proof}

Let $H$ be $\Col{\omega}{\nu}$-generic over $\VV[G]$.

\begin{subclaim*} 
 $\langle M,\in,I\rangle$ is $p[T]$-iterable in $\VV[G,H]$. 
\end{subclaim*} 

\begin{proof} 
We work in $\VV[G,H]$. Since there is a measurable cardinal in $M$, the structure $\langle M,\in,I\rangle$ is $\omega_1$-iterable by Lemma \ref{omega1 iterable structure}.  Set $B_{G*H}=p[T]^{\VV[G,H]}$. 
Suppose that $\map{j}{M}{M'}$ is a countable iteration. 
We argue that $p[j(T)]\cap M'= B_{G*H}\cap M'$. 
Since the statement $p[j(T)]\cap p[j(U)]\neq \emptyset$ is absolute between $M'$ and $\VV[G,H]$, this holds in $\VV[G,H]$. 
Since $p[T]\subseteq p[j(T)]$ and $p[U]\subseteq p[j(U)]$ and $p[T]\cup p[U]=\RRR$ in $\VV[G,H]$. 
This implies  $B_{G*H}\cap M'=p[T]^{\VV[G,H]} = p[j(T)]^{\VV[G,H]}$.  
\end{proof}

The existence of the required $B$-iterable structure is projective in $B$. 
Since results of Woodin show that our large cardinal assumption implies that the theory of $\LL(B,\RRR)$ cannot be changed by forcing (see {\cite[Theorem 3.3.8]{MR2069032}} and {\cite[Theorem 2.30]{MR1713438}}), such a $B$-iterable structure exists in $\VV$ and this structure witnesses that $x$ is an element of $Y$. 
\end{proof}

Since $Y$ is a $\Sigma^1_3(B)$-subset of $\RRR$, this completes the proof. 
\end{proof}

Note that it is also possible to prove the previous result in a similar way as Lemma \ref{lemma:HOmega2Sigma13} by using countable transitive models of $\ZFC^-$ with a Woodin cardinal that are \emph{$B$-iterable} (defined similar to Definition \ref{definition:Biterable}) with respect to the corresponding countable stationary tower forcing and its images. As in the above argument, the existence of such models witnessing $\Sigma_1$-statements is shown by  considering a suitable initial segment of $\VV$ in a generic extension by a collapse forcing and then using the same absoluteness argument as before.


\section{$\Sigma_1(\omega_1)$-definable subsets of ${}^{\omega_1}\omega_1$} \label{section: main results}

In this section, we present the proofs of the main results about $\Sigma_1(\omega_1)$-definable subset of $\HH{\kappa^+}$ stated in the introduction.

\subsection{Well-orderings of the reals} 

The above lemma directly yields the following strengthening of Theorem \ref{theorem:Omega1WO}. 

\begin{theorem}\label{NonExistenceWellOrder}
 Assume that either $M_1^\#(A)$ exists for every $A\subseteq\omega_1$ or that there is a precipitous ideal on $\omega_1$ and a measurable cardinal. If $A\subseteq\omega_1$ is $\mathbf{\Sigma}^1_2$ in the codes, then no well-ordering of the reals is $\Sigma_1(A)$-definable. 
\end{theorem}

\begin{proof}
 Assume that there is a well-ordering of the reals that is $\Sigma_1(A)$-definable. By Lemma \ref{lemma:HOmega2Sigma13} and Lemma \ref{characterization of Sigma1(omega1) from precipitous ideal}, this assumption implies that there is a $\mathbf{\Sigma}^1_3$-well-ordering of the reals.  
If $M_1^\#(A)$ exists for every $A\subseteq\omega_1$, then this yields a contradiction, because this assumption implies that $\mathbf{\Sigma}^1_2$-determinacy holds (see \cite{MR1349683}),  every $\mathbf{\Sigma}^1_3$-set of reals has the Baire property (see {\cite[6G.11]{MR2526093}}) and hence there are no $\mathbf{\Sigma}^1_3$-well-orderings of the reals.   
In the other case, if there is a precipitous ideal on $\omega_1$ and a measurable cardinal, then {\cite[Theorem 1.4]{MR576464}} shows that every $\mathbf{\Sigma}^1_3$-set of reals has the Baire property and we also derive a contradiction in this case. 
\end{proof}

We will consider $\Sigma_1$-well-orderings of the reals that allow more complicated parameters. 
As mentioned above, results of Woodin on the $\Pi_2$-maximality of the $\PPP_{max}$-extension of $\LL(\RRR)$ imply that no well-ordering of the reals is $\Sigma_1(A)$-definable over $\langle\HH{\omega_2},\in,B\rangle$ for some $A\in\POT{\omega_1}^{\LL(\RRR)}$ and $B\in\POT{\RRR}^{\LL(\RRR)}$.  
In the following, we will use $\PPP_{max}$-forcing to derive a stronger conclusion from a stronger assumption.

\begin{theorem} \label{no sigma1 wellorder from ub parameter} 
Suppose that there is a proper class of Woodin cardinals. 
If $B$ is uB, then there is no well-ordering of the reals which is $\Sigma_1(\omega_1)$-definable over the structure  $\langle\HH{\omega_2},\in,B,\mathrm{NS}_{\omega_1}\rangle$. 
\end{theorem} 

\begin{proof} 
If there is a proper class of Woodin cardinals, then every uB set of reals is determined by {\cite[Theorem 3.3.4 \& Theorem 3.3.14]{MR2069032}}. 
Hence the claim follows from Lemma \ref{characterization of Sigma1 with uB parameter}. 
\end{proof}


\subsection{Bernstein subsets} 

The next lemma shows how to construct perfect subsets of $\Sigma_1(\omega_1)$-definable subsets of ${}^{\omega_1}\omega_1$.  
It will allow us to prove that the existence of large cardinals implies the non-existence of $\Delta_1(\omega_1)$-definable Bernstein subsets of ${}^{\omega_1}\omega_1$. The lemma will also be used for a result about the non-stationary ideal (see Section \ref{subsection: club filter and non-stationary ideal}). In the following, we interpret a function $x\in{}^{\omega_1} \omega_1$ as a code for the subset $\bar{x}=\Set{\alpha<\omega_1}{x(\alpha)>0}$.

\begin{lemma}\label{lemma:Sigma1SetContainsBistationary}
 Assume that $M_1^\#(A)$ exists for every $A\subseteq\omega_1$. Let $A\subseteq\omega_1$ be $\mathbf{\Sigma}^1_2$ in the codes and let $X$ be a $\Sigma_1(A)$-definable subset of ${}^{\omega_1}\omega_1$. If there is an $x\in X$ with the property that $\bar{x}$ is a bistationary  subset of $\omega_1$, 
then for every $\xi<\omega_1$ there is 
\begin{enumerate} 
\item a continuous injection $\map{\iota}{{}^{\omega_1}2}{X}$ with  $\ran{\iota}\subseteq N_{x\restriction\xi}\cap X$ 

\item a club $D$ in $\omega_1$ with monotone enumeration $\seq{\delta_\alpha}{\alpha<\omega_1}$
\end{enumerate} 
such that for all $z\in{}^{\omega_1}2$ and $\alpha<\omega_1$, we have $z(\alpha)=1$ if and only if $\iota(z)(\delta_\alpha)>0$.  
\end{lemma}

\begin{proof}
 Fix $\xi<\omega_1$ and a $\Sigma_1$-formula $\varphi(v_0,v_1)$ with $X=\Set{z\in{}^{\omega_1}\omega_1}{\varphi(A,z)}$. Pick $a\in\RRR$ and a $\Sigma^1_2$-formula $\psi(v_0,v_1)$ with $A=\Set{\Vert w\Vert}{w\in\mathrm{WO}, ~ \psi(a,w)}$. 
 We can find $C\subseteq\omega_1$ such that $a,x,A\in M_1^{\#}(C)$,  $\omega_1=\omega_1^{M_1^{\#}(C)}$ and $\varphi(A,x)^{M_1^{\#}(C)}$. 
 Then  $\bar{x}$ is a bistationary subset of $\omega_1$ in $M_1^\#(C)$. 
 Note that every stationary subset of $\omega_1$ is a condition in $\mathbb{Q}_{<\delta}$. 
 Let $N$ be a countable elementary submodel of $M_1^\#(C)$ with $a,x,A\in N$ and $\xi+1\subseteq N$, let $\map{\pi}{N}{M}$ be the corresponding transitive collapse and let $\delta$ denote the unique Woodin cardinal in $M$. 
Since Lemma \ref{omega1 iterable structure} shows that $M$ is $\omega_1$-iterable with respect to $\mathbb{Q}_{<\delta}^M$ and its images, there is a directed system $$\langle\seq{M_s}{s\in{}^{{\leq}\omega_1}2}, ~ \seq{\map{j_{s,t}}{M_s}{M_t}}{s,t\in{}^{{\leq}\omega_1}2, ~ s\subseteq t}\rangle$$ of transitive models of $\ZFC^-$ and elementary embeddings such that the following statements hold. 
 \begin{enumerate}
  \item $M=M_\emptyset$.

  \item If $s\in{}^{{<}\omega_1}2$, then there are $M_s$-generic filters $G^s_0$ and $G^s_1$ over $j_{\emptyset,s}(\QQQ^M_{{<}\delta})$ such that $(j_{\emptyset,s}\circ\pi)(\omega_1\setminus\bar{x})\in G^s_0$, $(j_{\emptyset,s}\circ\pi)(\bar{x})\in G^s_1$, $M_{s^\frown\langle i\rangle}=\Ult{M_s}{G^s_i}$ and $j_{s,s^\frown\langle i\rangle}$ is the ultrapower map induced by $G^s_i$ for all $i<2$.

  \item If $s\in{}^{{\leq}\omega_1}2$ with $\length{s}\in\Lim$, then $$\langle M_s, ~ \seq{\map{j_{s\restriction\alpha,s}}{M_{s\restriction\alpha}}{M_s}}{\alpha<\length{s}}\rangle$$ is the direct limit of the directed system $$\langle\seq{M_{s\restriction\alpha}}{\alpha<\length{s}}, ~ \seq{\map{j_{s\restriction\bar{\alpha},s\restriction\alpha}}{M_{s\restriction\bar{\alpha}}}{M_{s\restriction\alpha}}}{\bar{\alpha}\leq\alpha<\length{s}}\rangle.$$ 
 \end{enumerate}
Let $j_s=j_{\emptyset,s}$ for all $s\in{}^{{\leq}\omega_1}2$.  
 Since $\omega_1=\omega_1^{M_z}$ for all $z\in{}^{\omega_1}2$, we can define $$\Map{i}{{}^{\omega_1}2}{{}^{\omega_1}\omega_1}{z}{(j_z\circ\pi)(x)}.$$  
Then elementarity and $\Sigma_1$-upwards absoluteness imply that $A=(j_z\circ\pi)(A)\in M_z$,  $x\restriction\xi=i(z)\restriction\xi$ and $\varphi(A,i(z))$ for all $z\in{}^{\omega_1}2$. This shows that $\ran{i}\subseteq N_{x\restriction\xi}\cap X$.

 Given $z\in{}^{\omega_1}2$, we define  $$\Map{c_z}{\omega_1}{\omega_1}{\alpha}{\omega_1^{M_{z\restriction\alpha}}}.$$ 
 By definition, $c_z$ is strictly increasing and continuous for every $z\in{}^{\omega_1}2$. Moreover, we have $c_{z_0}\restriction\alpha=c_{z_1}\restriction\alpha$  for all $z_0,z_1\in{}^{\omega_1}2$ and $\alpha<\omega_1$ with $z_0\restriction\alpha= z_1\restriction\alpha$.

 \begin{claim*} 
 Given $z\in {}^{\omega_1}2$ and $\alpha<\omega_1$, then $z(\alpha)=1$ if and only if  $i(z)(c_z(\alpha))>0$. 
 \end{claim*} 
 
 \begin{proof} 
 Given $z\in{}^{\omega_1}2$ and $\alpha<\omega_1$, we know that $c_z(\alpha)$ is smaller than the critical point of $j_{z\restriction(\alpha+1),z}$ and this allows us to use {\cite[Fact 2.7.3.]{MR2069032}} to conclude that   
 \begin{equation*}
  \begin{split}
   z(\alpha)=1 ~ &  \Longleftrightarrow ~ (j_{z\restriction\alpha}\circ\pi)(\bar{x})\in G^{z\restriction\alpha}_{z(\alpha)} \\
   & \Longleftrightarrow ~ \omega_1^{M_{z\restriction\alpha}}\in (j_{z\restriction(\alpha+1)}\circ\pi)(\bar{x}) \\
   & \Longleftrightarrow ~ c_z(\alpha)\in (j_{z\restriction(\alpha+1)}\circ\pi)(\bar{x}) \\
   & \Longleftrightarrow ~ (
((j_{z\restriction(\alpha+1)}\circ\pi)(x))(c_z(\alpha))>0 \\ 
   & \Longleftrightarrow ~ (
((j_{z\restriction(\alpha+1),z}\circ j_{z\restriction(\alpha+1)}\circ\pi)(x))(c_z(\alpha))>0 \\ 
  & \Longleftrightarrow ~ i(z) (c_z(\alpha))>0. \qedhere
  \end{split} 
 \end{equation*}
 \end{proof} 
 In particular, this shows that the function $i$ is injective.

\begin{claim*} 
The function $i$ is continuous. 
\end{claim*} 
\begin{proof} 
Let $z\in{}^\kappa 2$ and $\beta<\kappa$. Then there is  $\alpha<\kappa$  with $\beta<c_z(\alpha)<\crit{j_{z}}$. Given $\bar{z}\in{}^{\omega_1}2$, we know that $c_{\bar{z}}(\alpha)$ is the critical point of $j_{\bar{z}\restriction\alpha,z}$ and hence $$i(\bar{z})\restriction\beta ~ = ~ (j_{\bar{z}}\circ\pi)(x)\restriction\beta ~ = ~ (j_{\bar{z}\restriction\alpha}\circ\pi)(x)\restriction\beta.$$ If  $\bar{z}\in N_{z\restriction\alpha}\cap{}^{\omega_1}2$, then $j_{z\restriction\alpha}=j_{\bar{z}\restriction\alpha}$ and therefore $i(z)\restriction\beta=i(\bar{z})\restriction\beta$. 
\end{proof}

\begin{claim*} 
There is a club $D$ in $\omega_1$ such that $c_z\restriction D=\id_D$ for all $z\in {}^{\omega_1}2$. 
\end{claim*} 

\begin{proof} 
Suppose that $z_M$ is a real coding $M$. We define $D=\mathrm{Card}^{\LL[z_M]}\cap\omega_1$. 
A statement and proof analogous to {\cite[Lemma 19]{MR2963017}} for forcing with $\mathbb{Q}_{<\delta}$ instead of a precipitous ideal shows that the cardinals in $\LL[z_M]$ are closure points of the images of $c_{z}$ for all $z\in {}^{\omega_1}2$. We can conclude that $c_z\restriction D=\id_D$ for all $z\in {}^{\omega_1}2$. 
\end{proof} 

Let $\seq{\delta_\alpha}{\alpha<\omega_1}$ denote the monotone enumeration of $D$ and let $\map{e}{{}^{\omega_1}2}{{}^{\omega_1}2}$ denote the unique continuous injection with $e(z)^{{-}1}\{1\}=\Set{\delta_\alpha}{\alpha<\omega_1, ~ z(\alpha)=1}$ for all $z\in{}^{\omega_1}2$. Set $\iota=i\circ e$. Then we have 
 \begin{equation*}
  z(\alpha)=1 ~ \Longleftrightarrow ~ e(z)(\delta_\alpha)=1 ~ \Longleftrightarrow ~ i(e(z)) (c_{e(z)}(\delta_\alpha))>0 ~ \Longleftrightarrow ~ \iota(z)(\delta_\alpha)>0  
 \end{equation*} 
for all $z\in{}^{\omega_1}2$ and $\alpha<\omega_1$. 
\end{proof}

A simpler version of the proof of Lemma \ref{lemma:Sigma1SetContainsBistationary} shows the following.

\begin{lemma} \label{Bernstein from precipitous ideal} 
The conclusion of Lemma \ref{lemma:Sigma1SetContainsBistationary} follows from the existence of a precipitous ideal on $\omega_1$ and a measurable cardinal. \qed 
\end{lemma}

The above lemmas allow us to prove the following strengthening of Theorem \ref{theorem:Omega1Bernstein}.

\begin{theorem}\label{theorem:M1SharpBernstein}
 Assume that either $M_1^\#(A)$ exists for every $A\subseteq\omega_1$ or that there is a precipitous ideal on $\omega_1$ and a measurable cardinal. 
 Let $\Gamma$ denote the collection of subsets of ${}^{\omega_1}\omega_1$ that are $\Sigma_1(A)$-definable for some $A\subseteq\omega_1$ that is $\mathbf{\Sigma}^1_2$ in the codes. 
 If $\Delta\subseteq\Gamma$ with $\bigcup \Delta= {}^{\omega_1}\omega_1$, then some element of $\Delta$ contains a perfect subset. \end{theorem}

\begin{proof} 
 Pick some $x\in{}^{\omega_1}\omega_1$ with the property that $\Set{\alpha<\omega_1}{x(\alpha)=1}$ is a bistationary subset of $\omega_1$.   Then there is $X\in\Delta$ with $x\in X$.  In this situation, Lemma \ref{lemma:Sigma1SetContainsBistationary} and Lemma \ref{Bernstein from precipitous ideal} imply that $X$ contains a perfect subset. 
\end{proof}

\begin{theorem} \label{no Delta1-definable Bernstein set} 
 Assume that either $M_1^\#(A)$ exists for every $A\subseteq\omega_1$  or that there is a precipitous ideal on $\omega_1$ and a measurable cardinal. If $A\subseteq\omega_1$ is $\mathbf{\Sigma}^1_2$ in the codes, then no Bernstein subset of ${}^{\omega_1}\omega_1$ is $\Delta_1(A)$-definable over $\langle\HH{\omega_2},\in\rangle$.  
\end{theorem} 

\begin{proof} 
 Apply Theorem \ref{theorem:M1SharpBernstein} with $\Delta=\{A, ~ {\omega_1\setminus A}\}\subseteq\Gamma$. 
\end{proof}

We will see in Lemma \ref{Delta1 Bernstein set in M_1} below that the existence of a $\Sigma_1(\omega_1)$-definable Bernstein subset of ${}^{\omega_1}\omega_1$ is consistent with the existence of a Woodin cardinal.


\subsection{A perfect set theorem}\label{subsection:PerfectSet}

We aim to prove a perfect set theorem for $\Sigma_1(\omega_1)$-definable subsets of ${}^{\omega_1}\omega_1$. This is motivated by the following result.

\begin{theorem}[Woodin, {\cite[Corollary 7.11]{MR2768703}}]
Assume $\mathsf{AD}^{\LL(\RRR)}$ and suppose that $G$ is $\PPP_{max}$-generic over $\LL(\RRR)$. 
Work in $\LL(\RRR)[G]$. 
Suppose that $A$ is a subset of ${}^{\omega_1}\omega_1$ which is defined from a parameter in $\LL(\RRR)$. Then at least one of the following conditions hold. 
\begin{enumerate} 
\item 
$A$ contains a perfect subset. 
\item 
$A\subseteq\LL(\RRR)$. 
\end{enumerate} 
\end{theorem} 

We will prove a similar result for $\Sigma_1(\omega_1)$-definable sets in $\VV$ from the assumption that $\mathrm{NS}_{\omega_1}$ is saturated and there is a measurable cardinal. We do not know if our result is a true dichotomy, i.e. whether the two cases are mutually exclusive. 

Assuming that $\mathrm{NS}_{\omega_1}$ is saturated, the following result of Woodin shows that there is a canonical iteration of length $\omega_1$ of any countable substructure of $\HH{\omega_2}$.

\begin{lemma} [Woodin] \label{canonical iteration}
Suppose that the non-stationary ideal $\mathrm{NS}_{\omega_1}$ on $\omega_1$ is saturated.  
If $A\subseteq\omega_1$ and $\map{i}{\langle M,\in,I,\bar{A}\rangle}{\langle\HH{\theta},\in,\mathrm{NS}_{\omega_1},A\rangle}$ is an elementary embedding  with $\theta\geq\omega_2$ and $M$ is countable, then there is a generic iteration $\map{j}{M}{N}$ of length $\omega_1$  with $N$ well-founded and $j(\bar{A})=A$. 
\end{lemma} 

\begin{proof} 
 We inductively construct a generic iteration $$\langle \seq{\langle M_\alpha,\in, I_\alpha, \bar{A}_\alpha\rangle}{\alpha<\omega_1},\seq{\map{ i_{\alpha,\beta}}{M_\alpha}{M_\beta}}{\alpha\leq \beta<\omega_1}\rangle$$  with $M=M_0$ and elementary embeddings $\seq{\map{j_{\alpha}}{M_\alpha}{M}}{\alpha<\omega_1}$ such that $j_\alpha= j_\beta \circ i_{\alpha,\beta}$ for all $\alpha\leq\beta<\omega_1$. 
Suppose that $\langle M_\alpha,\in, I_\alpha, \bar{A}_\alpha\rangle$, $i_{\alpha,\beta}$ and $j_\alpha$ are defined for  $\alpha\leq\beta\leq\gamma$.  
Set $\kappa=i_{0,\gamma}(\omega_1^{M_\gamma})$ and $U_\gamma=\Set{X\in \POT{\kappa}^{M_\gamma}}{\omega_1\in j_{\gamma}(X)}$. 

\begin{claim*} 
$U_\gamma$ is $\POT{\kappa}/I_\gamma$-generic over $M_\gamma$. 

\end{claim*} 
\begin{proof} 
Suppose that $A\in M_\gamma$ is a maximal antichain in $\POT{\kappa}/I_\gamma$. 
Since $\mathrm{NS}_{\omega_1}$ is saturated, $\POT{\kappa}/I_\gamma$ satisfies the $\omega_2^{M_\gamma}$-chain condition in $M_\gamma$.  
Let $\langle X_\alpha\mid \alpha<\kappa\rangle$ enumerate $A$ in $M_\gamma$ and assume that $X_\alpha \notin U_\gamma$ for all $\alpha<\kappa$.  
By the definition of $U_\gamma$, we have $X=\bigtriangleup_{\alpha<\kappa} (\kappa\setminus X_\alpha)\in U_\gamma$ . 
Since $U_\gamma$ is normal, the set $X$ is stationary. 
This contradicts the assumption that $A$ is maximal. 
\end{proof} 

We define $M_{\gamma+1}=\Ult{M_\gamma}{U_\gamma}$, $\map{i_{\gamma,\gamma+1}}{M_\gamma}{M_{\gamma+1}}$ the ultrapower map, and $\map{j_{\gamma+1}}{M_{\gamma+1}}{\HH{\theta}}$ by $j_{\gamma+1}([f])=j_\gamma(f)(\omega_1)$. 
It is straightforward to check that $j_{\gamma+1}$ is well-defined and elementary. 

\begin{claim*} 
$j_{\gamma}= j_{\gamma+1} \circ i_{\gamma,\gamma+1}$. 
\end{claim*} 
\begin{proof} 
If $x\in M_\gamma$, then 
 \begin{equation*}
  j_{\gamma+1}(i_{\gamma,\gamma+1}(x)) ~ = ~ j_{\gamma+1}([c_x]) ~ = ~ j_{\gamma+1}(c_x)(\omega_1) ~ = ~ c_{j_\gamma(x)}(\omega_1) ~ = ~ j_\gamma(x). \qedhere
 \end{equation*} 
\end{proof} 
This completes the proof of the lemma. 
\end{proof}

\begin{theorem} \label{perfect subsets} 
Suppose that $\mathrm{NS}_{\omega_1}$ is saturated and there is a measurable cardinal.  Suppose that $X$ is a $\Sigma_1(\omega_1)$-definable subset of ${}^{\omega_1}\omega_1$. 
Then at least one of the following conditions holds. 
\begin{enumerate} 
\item 
$X$ contains a perfect subset. 
\item 
$X \subseteq\LL(\RRR)$. 
\end{enumerate} 
\end{theorem} 
\begin{proof} 
Suppose that $\mu$ is measurable and $\theta=\mu^+$. 
Suppose that $X\not\subseteq\LL(\RRR)$. Then there is some $A\in X\setminus\LL(\RRR)$. 
Suppose that $\map{i}{\langle M,\in,I,\bar{A}\rangle}{\langle\HH{\theta},\in, \mathrm{NS}_{\omega_1},A\rangle}$ is elementary and $M$ is countable. Let $\bar{\mu}=i^{-1}(\mu)$. Since $\mathrm{NS}_{\omega_1}$ is saturated and $\POT{\omega_1}^{\#}$ exists, 
$\langle M,\in,I,\bar{A}\rangle$ is $\omega_1$-iterable by {\cite[Theorem 3.10 \& Theorem 4.29]{MR1713438}}.

\begin{claim*} 
 If $i_0(\bar{A})\cap \alpha = i_1(\bar{A})\cap \alpha$ with  $\alpha=\min(\{i_0(\omega_1^M), i_1(\omega_1^M)\})$ holds for all countable iterations $\map{i_0}{M}{N_0}$ and $\map{i_1}{M}{N_1}$, then $X\subseteq\LL(\RRR)$. 
\end{claim*} 

\begin{proof} 
It follows from Lemma \ref{canonical iteration} that $i_0(\bar{A})\cap \alpha= A\cap \alpha$. 
Hence $A$ can be reconstructed from $(M,I,\bar{A})$ in $\LL(\RRR)$ by considering generic iterations of arbitrarily large countable length in $\LL(\RRR)$. 
\end{proof}

\begin{claim*} 
If there are countable iterations $\map{i_0}{M}{N_0}$ and $\map{i_1}{M}{N_1}$ with $i_0(\bar{A})\cap \alpha \neq i_1(\bar{A})\cap \alpha$ for $\alpha=\min(\{i_0(\omega_1^M), i_1(\omega_1^M)\})$, then this remains true in every countable iterate of $M$.  
\end{claim*} 

\begin{proof} 
Let $\gamma=\max(\{i_0(\omega_1^M), i_1(\omega_1^M)\})$. 
Suppose that $\bar{U}$ is a normal measure on $\bar{\mu}$ in $M$. 
Suppose that $\map{j}{M}{M^\gamma}$ is the iterate of $M$ of length $\gamma$ with $\bar{U}$.  
Then $j(\bar{\mu})>\gamma$. As in the proof of Lemma \ref{omega1 iterable structure}, the iterated ultrapowers of $M$ with $\bar{U}$ commute with the generic ultrapower since $\bar{\mu}>(2^{\omega_1})^M$. The same argument works for all further steps in the generic iteration of $M$ and hence we obtain a commutative diagram. This shows that the generic iteration of $M^\gamma$ commutes with the generic iteration of $M$.  
In any $\Col{\omega}{j(\gamma)}$-generic extension of $M^{\gamma}$, there are sequences of ultrafilters which induce $i_0, i_1$ as in the statement of the claim by $\mathbf{\Sigma}^1_2$-absoluteness. 
Hence such iterations exist in any $\Col{\omega}{\gamma}$-generic extension of $M$ by elementarity. 
This statement is preserved in generic iterations of $M$ by elementarity and guarantees the existence of $i_0$ and $i_1$. 
\end{proof} 

The last claim allows us to build a perfect tree $T$ of height $\omega_1$ of generic iterates of $M$ with the property that the set of images of $\bar{A}$ along the branches of $T$ form a perfect subset of $X$.  
\end{proof}

\begin{remark} 
 If $\mathsf{CH}$ fails, then the set $X=\Set{x\in {}^{\omega_1}\omega_1}{\forall \alpha\geq\omega\ ~ x(\alpha)=0}$  is a $\Delta_1(\omega_1)$-definable subset of ${}^{\omega_1}\omega_1$ without the perfect set property.  
\end{remark}


\subsection{The club filter and the non-stationary ideal} \label{subsection: club filter and non-stationary ideal} 

In this section, we will use Lemma \ref{lemma:Sigma1SetContainsBistationary} to prove a strengthening of Theorem \ref{theorem:Omega1Club}. 

\begin{lemma}\label{lemma:Omega1DenseClubNS} 
 Assume that either $M_1^\#(A)$ exists for every $A\subseteq\omega_1$ or that there is a precipitous ideal on $\omega_1$ and a measurable cardinal. Let $A$ be an unbounded subset of $\omega_1$ that is $\mathbf{\Sigma}^1_2$ in the codes and let $Y$ be a $\Sigma_1(A)$-definable subset of $\POT{\omega_1}$.  Then the following statements hold for all $y\in Y$ and $\xi<\omega_1$. 
 \begin{enumerate}
  \item If $y$ is a stationary subset of $\omega_1$, then there is $z\in Y$ such that $z$ is an element of the club filter on $\omega_1$ and $y\cap\xi=z\cap\xi$. 
  
  \item If $y$ is a costationary subset of $\omega_1$, then there is $z\in Y$ such that $z$ is an element of the nonstationary ideal on $\omega_1$  and $y\cap\xi=z\cap\xi$. 
 \end{enumerate}
\end{lemma}

\begin{proof}
 Let $X\subseteq{}^{\omega_1}2$ denote the set of  characteristic functions of elements of the set $Y$. Since $A$ is unbounded in $\omega_1$, the set $X$ is $\Sigma_1(A)$-definable. Fix $y\in Y$ and $\xi<\omega_1$. In the following, we may assume that $y$ is a bistationary subset of $\omega_1$, because otherwise the above statements hold trivially. Let $x\in X$ denote the characteristic function of $y$.  
 We can apply Lemma \ref{lemma:Sigma1SetContainsBistationary} and Lemma \ref{Bernstein from precipitous ideal} to find $x_0,x_1\in N_{x\restriction\xi}\cap X$ and a monotone enumeration $\langle c_{\alpha}\mid \alpha<\omega_1\rangle$ of a club $C$ in $\omega_1$ such that $x_i(c_\alpha)=i$ for all $\alpha<\omega_1$ and $i<2$. Set $z_i=\Set{\alpha<\omega_1}{x_i(\alpha)>0}\in Y$ for $i<2$.  
Then $C$ witnesses that $z_0$ is an element of the club filter on $\omega_1$ and that $z_1$ is an element of the nonstationary ideal on $\omega_1$. 
\end{proof}

\begin{theorem} \label{separate club filter and nonstationary ideal} 
 Assume that either $M_1^\#(A)$ exists for every $A\subseteq\omega_1$ or that there is a precipitous ideal on $\omega_1$ and a measurable cardinal. If $A\subseteq\omega_1$ is $\mathbf{\Sigma}^1_2$ in the codes and $X$ is a subset of $\POT{\omega_1}$ that separates the club filter from the non-stationary ideal, then $X$ is not $\Delta_1(A)$-definable.    
\end{theorem}
\begin{proof}
 Assume that the set $X$ is $\Delta_1(A)$-definable over $\langle\HH{\omega_2},\in\rangle$. 
 Since $X$ is disjoint from the nonstationary ideal on $\omega_1$ and therefore contains no countable subsets of $\omega_1$, $\Sigma_1$-reflection implies that $A$ is unbounded in $\omega_1$ and the second part of Lemma \ref{lemma:Omega1DenseClubNS} shows that $X$ contains no costationary subsets of $\omega_1$. But this implies that $X$ is equal to the club filter on $\omega_1$ and therefore $\POT{\omega_1}\setminus X$ contains a stationary subset of $\omega_1$. In this situation, the first part of  Lemma \ref{lemma:Omega1DenseClubNS} implies that $\POT{\omega_1}\setminus X$ contains an element of the club filter on $\omega_1$, a contradiction.  
\end{proof}

\begin{corollary}
 Assume that either $M_1^\#(A)$ exists for every $A\subseteq\omega_1$ or that there is a precipitous ideal on $\omega_1$ and a measurable cardinal. If $A\subseteq\omega_1$ is $\mathbf{\Sigma}^1_2$ in the codes, then the club filter on $\omega_1$ is not $\Pi_1(A)$-definable over $\langle\HH{\omega_2},\in\rangle$. 
\end{corollary}

\begin{proof} 
This is immediate from Theorem \ref{separate club filter and nonstationary ideal}. 
\end{proof}

We can also use Lemma \ref{lemma:Omega1DenseClubNS} to study  $\Sigma_1(\omega_1)$-definable singletons.

\begin{lemma}
 Assume that either $M_1^\#(A)$ exists for every $A\subseteq\omega_1$ or that there is a precipitous ideal on $\omega_1$ and a measurable cardinal. If $A\subseteq\omega_1$ is $\mathbf{\Sigma}^1_2$ in the codes and $x$ is a subset of $\omega_1$ with the property that $\{x\}$ is $\Sigma_1(A)$-definable,  then $x$ is either contained in the club filter on $\omega_1$ or in the nonstationary ideal on $\omega_1$.  
\end{lemma}

\begin{proof}
 If $A$ is bounded in $\omega_1$, then $\Sigma_1$-reflection implies that $x\in\HH{\omega_1}$ and hence $x$ is contained in the nonstationary ideal on $\omega_1$. Otherwise $A$ is unbounded in $\omega_1$ and the claim follows directly from Lemma \ref{lemma:Omega1DenseClubNS}. 
\end{proof}

\begin{remark}
 If $\VV=\LL$ and $\kappa$ is an uncountable regular cardinal, then there is a bistationary subset $x$ of $\kappa$ such that $\{x\}$ is $\Sigma_1(\kappa)$-definable. Such subsets can be constructed from the  canonical $\Diamond_\kappa$-sequence in $\LL$,  using the facts that this sequence is definable over $\langle\LL_\kappa,\in\rangle$ by a formula without parameters and the set $\{\LL_\kappa\}$ is $\Sigma_1(\kappa)$-definable. Another way to construct such subsets is described in {\cite[Section 7]{MR3591274}}.
\end{remark}


\subsection{Uniformization of the club filter}\label{subsection:Uniformizations of the club filter}

We show that  the existence of large cardinals implies that the club filter on $\omega_1$ has no $\Sigma_1(\omega_1)$-definable uniformization.

\begin{definition} 
 Let $\kappa$ be an uncountable regular cardinal. 
A \emph{uniformization} of the club filter  on $\kappa$ is a function $\map{f}{\mathrm{C}_\kappa}{\mathrm{C}_\kappa}$ such that $f(X)\subseteq X$ is a club for all $X\in \mathrm{C}_\kappa$. 
\end{definition}

\begin{lemma} \label{club uniformization} 
If in a model of $\mathsf{ZF}$, the club filter $\mathrm{C}_{\omega_1}$ on $\omega_1$ is an ultrafilter, then there is no uniformization of $\mathrm{C}_{\omega_1}$ which is definable from a set of ordinals.  
\end{lemma} 

\begin{proof} 
Suppose that the club filter $\mathrm{C}_{\omega_1}$ is an ultrafilter and there is a uniformization of $\mathrm{C}_{\omega_1}$ which is definable from a set of ordinals $z$. 
Then we can find a function $\map{f}{\POT{\omega_1}}{\mathrm{C}_{\omega_1}}$ definable from $z$ such that for all $A\in\POT{\omega_1}$, $f(A)$ is a club subset of $A$ or of its complement. 
Let $\mathrm{HOD}_z$ denote the class of sets which are hereditarily ordinal definable from $z$. Since $\omega_1$ is regular in $\mathrm{HOD}_z$, there is a subset of $\omega_1$ which is bistationary in $\mathrm{HOD}_z$. The least such set $S$ in a definable enumeration of $\mathrm{HOD}_z$ is definable from $z$ and $\omega_1$.  
Then $f(S)\in \mathrm{HOD}$ and hence $S$ is not bistationary in $\mathrm{HOD}$.  
\end{proof}

\begin{remark} 
Suppose that in a model of $\mathsf{ZF}$, $x^{\#}$ exists for every real $x$ (and hence for every $x\in [\omega_1]^{<\omega_1}$), and there is no uniformization of $\mathrm{C}_{\omega_1}$. Then there is no function $\map{f}{\POT{\omega_1}}{[\omega_1]^{<\omega_1}}$ such that $A\in L[f(A)]$ for all $A\subseteq\omega_1$. 
Suppose that $f$ is such a function. For $A\subseteq\omega_1$ let $x_A$ denote the inclusion-least finite set of $f(A)$-indiscernibles such that $A$ is definable from $f(A)$ and $x_A$ in $L[f(A)]$. 
Then the club $C_A$ of $f(A)$-indiscernibles (i.e. Silver indiscernibles) between $\sup(x_A\cap \omega_1)$ and $\omega_1$ is either contained in $A$ or disjoint from $A$. 
Since $C_A$ is definable from $f(A)^{\#}$, this defines a uniformization of $\mathrm{C}_{\omega_1}$, contradicting the assumption. 
\end{remark} 

\begin{theorem} \label{Sigma1-definable club uniformization} 
Suppose that there are infinitely many Woodin cardinals and a measurable cardinal above them. 
\begin{enumerate} 
 \item  In $\LL(\RRR)$, there is no uniformization of the club filter on $\omega_1$.  

 \item There is no $\Sigma_1(\omega_1)$-definable uniformization of the club filter on $\omega_1$.  
\end{enumerate} 
\end{theorem} 

\begin{proof} 
 (i) In $\LL(\RRR)$, every element is ordinal definable from a real and our assumptions imply that the club filter on $\omega_1$ is an ultrafilter. By Lemma \ref{club uniformization}, there is no uniformization of the club filter on $\omega_1$.

 (ii) Assume that there is a $\Sigma_1(\omega_1)$-definable uniformization of $\mathrm{C}_{\omega_1}$. 
By the $\Pi_2$-maximality of the $\PPP_{max}$-extension of $\LL(\RRR)$ (see {\cite[Theorem 7.3]{MR2768703}}), the same $\Sigma_1$-formula defines a uniformization of $\mathrm{C}_{\omega_1}$ in the $\PPP_{max}$-extension of $\LL(\RRR)$. Since $\PPP_{max}$ is weakly homogeneous in $\LL(\RRR)$ (see {\cite[Lemma 2.10]{MR2768703}}), this shows that there is a uniformization of $\mathrm{C}_{\omega_1}$ in $\LL(\RRR)$, contradicting the first part of the theorem.   
\end{proof} 

\begin{remark} 
 Unpublished results of Woodin (see {\cite[Remark 3.3.12]{MR2069032}} and {\cite[End of Section 6.3]{LarsonBriefHistoryOfDeterminacy}}) show that the existence of a proper class of Woodin limits of Woodin cardinals implies that the axiom of determinacy holds in the Chang model $\LL(\On^{\omega})$. Hence $\mathrm{C}_{\omega_1}$ is an ultrafilter in $\LL(\On^{\omega})$. 
It follows from Lemma \ref{club uniformization} that there is no uniformization of $\mathrm{C}_{\omega_1}$ in $\LL(\On^{\omega})$.  
\end{remark}

\begin{remark}
Let $\kappa$ be inaccessible in $\LL$ and let $G$ be $\Col{\omega}{{<}\kappa}$-generic over $\LL$. Since $\Col{\omega}{{<}\kappa}$ satisfies the $\kappa$-chain condition in $\LL$, every element of $\mathrm{C}_{\omega_1}^{\LL[G]}$ contains a constructible club and there is a uniformization of $\mathrm{C}_{\omega_1}$ in $\LL(\RRR)^{\LL[G]}$. 
\end{remark}


\subsection{$\Sigma_1(\omega_1)$-absoluteness}\label{subsection:absoluteness}

In this section, we observe that for $\Sigma_1(\omega_1)$-formulas, absoluteness to $\omega_1$-preserving forcings holds for formulas without parameters, but not for formulas with subsets of $\omega_1$ as parameters.

\begin{lemma} \label{Sigma1 absoluteness} 
 Let $\delta$ be a Woodin cardinal below a measurable cardinal. 
 \begin{enumerate} 

  \item $\Sigma_1(\omega_1)$ statements (without parameters) are absolute to generic extensions for forcings of size less than $\delta$.\footnote{Given a $\Sigma_1$-formula $\varphi(v)$, a partial order $\PPP$ of cardinality less than $\delta$ and $G$ $\PPP$-generic over $\VV$, then this statement says that  $\varphi(\omega_1^\VV)^\VV$ holds if and only if $\varphi(\omega_1^{\VV[G]})^{\VV[G]}$ holds.}  

\item The set of $\Sigma_1(\omega_1)$-formulas defining sets $\{x\}$ with $x\subseteq\omega$ is absolute for forcings of size less than $\delta$. Moreover, 
the set of $\Sigma_1(\omega_1)$-definable singletons $\{x\}$ with $x\subseteq\omega_1$ is absolute for $\omega_1$-preserving forcings of size less than $\delta$. 

 \item  The canonical code for $M_1^{\#}$ is a subset of $\omega$ which is not $\Sigma_1(\omega_1)$-definable in any generic extension by forcings of size less than $\delta$. 
\end{enumerate} 
\end{lemma} 

\begin{proof} 
 The first statement follows directly from Lemma \ref{lemma:HOmega2Sigma13}, since it is equivalent to a $\Sigma^1_3$-statement. 
The second statement follows from the first statement.  
For the third statement, suppose that the canonical code for $M_1^{\#}$ is $\Sigma_1(\omega_1)$-definable. Then it is $\Sigma^1_3$-definable by Lemma \ref{lemma:HOmega2Sigma13}. 
It is well known that forcing of size less than $\delta$ preserves $M_1^{\#}$ (see {\cite[Lemma 3.7]{MR3226056}}).  Since $\Sigma^1_3$-truth can be computed in $M_1^{\#}$ (see {\cite[p. 1660]{MR2768698}}), the canonical code for $M_1^{\#}$ is an element of $M_1^{\#}$, a contradiction.  
\end{proof}

\begin{remark} 
 The existence of large cardinals does not imply that $\Sigma_1(\omega_1)$-formulas with parameters in $\HH{\omega_2}$ are absolute to generic extensions which preserve $\omega_1$. 
For instance, we can add a Suslin tree $T$ by adding a Cohen real (see {\cite[Theorem 28.12]{MR1940513}}).  
When we add a branch through $T$ by forcing with $T$, $\omega_1$ is not collapsed. Note that the existence of a branch through $T$ is $\Sigma_1(T)$. 
\end{remark}


\section{$\Sigma_1(\omega_1)$-definable sets in $M_1$}\label{section: M1} 

We show that for some of the results above, large cardinal assumptions are necessary, because these  results fail in $M_1$. We start by showing that the assumption of Theorem \ref{NonExistenceWellOrder} is optimal. For other applications, we will construct well-orderings of $\HH{\kappa^+}$ with the property that the initial segments are uniformly $\Sigma_1(\kappa)$-definable.

\begin{definition} 
 Given an infinite cardinal $\kappa$, a well-ordering $\lhd$ of a subset of $\HH{\kappa^+}$ is a \emph{good $\Sigma_1(\kappa)$-well-ordering} if the set $I(\lhd)=\Set{\Set{x}{x\lhd y}}{y\in\ran{\lhd}}$ of all proper initial segments of $\lhd$ is $\Sigma_1(\kappa)$-definable.  
\end{definition}

\begin{theorem}\label{Sigma1 wellorder in M1} 
Suppose that $M_1$ exists. 
In $M_1$, the canonical well-ordering of $M_1$ restricted to $\HH{\omega_2}$ is a good $\Sigma_1(\omega_1)$-definable well-order.  
\end{theorem}

\begin{proof}  
 Let $\delta$ be the unique Woodin cardinal in $M_1$. 
 Work in $M_1\vert\delta$. Then there is no inner model with a Woodin cardinal, because $M_1\vert\delta$ is closed under sharps and, by a theorem of Woodin, the existence of such an inner model would imply that $M_1^\#$ is an element of $M_1\vert\delta$.\footnote{This result is unpublished, but the methods used in the (known) proof can be found in \cite{CabalVolume3}.} 
 
  By a \emph{mouse} we mean a premouse in the sense of Mitchell-Steel \cite{MR1300637} such that all countable elementary substructures are $\omega_1$-iterable. 
 The previous argument allows us to use {\cite[Lemma 2.1]{MR2963017}} to conclude that a premouse $M\in\HH{\omega_2}$ with no definable Woodin cardinals is a mouse if and only if there is a transitive model $U\in\HH{\omega_2}$ of $\ZFC^-$ plus \anf{{there is no inner model with a Woodin cardinal}} with $\omega_1\subseteq U$ and $\langle U,\in\rangle\models \anf{M \textit{ is a mouse}}$. This shows that the set 
$$ A ~ = ~ \Set{M\in\HH{\omega_2}}{\text{{$M$ is a mouse}, $\omega_1^{M} = \omega_1$, $\rho_\omega(M)=\omega_1$}}$$
is $\Sigma_1(\omega_1)$-definable. 
Since $N\in\HH{\omega_2}$ is an initial segment of $M_1\vert\omega_2$ if and only if $N$ is a proper initial segment of some $M$ in $A$, the above computations show that the collection of all initial segments of $M_1\vert\omega_2$ is also $\Sigma_1(\omega_1)$-definable.

 Let $\lhd$ denote the canonical well-ordering of $\HH{\omega_2}$ in $M_1$. Given $x,y\in\HH{\omega_2}$, we have $x\lhd y$ if and only if there is an initial segment $N$ of $M_1\vert\omega_2$
 such that $x,y\in N$ and $x<_N y$, where $<_N$ is the canonical well-ordering of $N$. By the above computations, this shows that $\lhd$ is a good $\Sigma_1$-definable well-order of $H(\omega_2)^{M_1}$.  
\end{proof}

\begin{theorem} \label{Sigma1 wellorder and non CH} 
 Suppose that $M_1$ exists. There is a generic extension of $M_1$ in which $\neg\mathsf{CH}$ holds and there is a good $\Sigma_1(\omega_1)$-definable well-order of $H(\omega_2)$.  
\end{theorem}

\begin{proof} 
 Let $\delta$ denote the unique Woodin cardinal in $M_1$.  Work in $M_1$ and let $\lhd$ denote the canonical well-ordering of $M_1$. Given $\alpha\in\omega_1\cap\Lim$, let $C_\alpha$ denote the $\lhd$-least cofinal subset of $\alpha$ of order-type $\omega$. Then $\vec{C}=\seq{C_\alpha}{\alpha\in\omega_1\cap\Lim}$ is a $C$-sequence. 
 Let $\nu<\delta$ be a Mahlo cardinal and let $\kappa<\nu$ be $\Sigma_1$-reflecting in $M_1\vert\nu$. In this situation, let $\PPP$ denote the partial order constructed in \cite{MR1324501} that forces $\mathrm{BPFA}$ to hold in a generic extension of $M_1\vert\nu$ using the reflecting cardinal $\kappa$ and let $G$ be $\PPP$-generic over $M_1$.  Then $\omega_1^{M_1}=\omega_1^{M_1[G]}$, $\HH{\omega_2}^{(M_1\vert\nu)[G]}=\HH{\omega_2}^{M_1[G]}$, $\vec{C}$ is still a $C$-sequence in $(M_1\vert\nu)[G]$ and, by {\cite[Theorem 2]{MR2231126}}, there is a good $\Sigma_1(\vec{C})$-definable well-ordering of $\HH{\omega_2}$ in $M_1[G]$.  
  The forcing does not add an inner model with a Woodin cardinal, since (as in the proof of Lemma \ref{Sigma1 wellorder in M1}) this would imply that $M_1^{\#}$ is an element of $(M_1|\delta)[G]$ and hence of $M_1|\delta$, by using two mutual generics and the fact that $M_1$ is $\Sigma^1_3$-correct in $\VV$. 
 Hence we can use the same $\Sigma_1(\omega_1)$-definition of the initial segments of $M_1$ as in the proof of Lemma \ref{Sigma1 wellorder in M1}. 
 Therefore the set $\{\vec{C}\}$ is $\Sigma_1(\omega_1)$-definable in $M_1[G]$.  This yields the statement of the theorem. 
\end{proof}

\begin{theorem}
Suppose that $M_1$ exists. Then the following statements hold in a forcing extension $M_1[G]$ of $M_1$. 
 \begin{enumerate}
   \item There is a Woodin cardinal.

    \item The $\GCH$ fails at $\omega_1$. 

   \item There is a $\Sigma_1(\omega_1)$-definable well-ordering of $\HH{\omega_2}$. 
   \end{enumerate}
\end{theorem}

\begin{proof} 
 If $\delta$ is the unique Woodin cardinal in $M_1$ and $\lhd$ is the canonical well-ordering of $M_1$ restricted to $\HH{\omega_2}^{M_1}$, then the following statements hold in $M_1$: 

  \begin{enumerate}
  \item $\lhd$ is a good $\Sigma_1(\omega_1)$-definable well-ordering.  
  
  \item If $\PPP$ is a partial order of cardinality less than $\delta$ with the property that forcing with $\PPP$ preserves cofinalities less than or equal to $\omega_2$  and $G$ is $\PPP$-generic over $\VV$, then $\HH{\omega_2}^\VV$ is $\Sigma_1(\omega_1)$-definable in $\VV[G]$. 

  \item There is a closed unbounded subset of $[\HH{\omega_2}]^\omega$ consisting of elementary submodels $M$ of $\HH{\omega_2}$ with $\pi[I(\lhd)\cap M]\subseteq I(\lhd)$, where $\map{\pi}{M}{N}$ denotes the corresponding transitive collapse. 
   \end{enumerate}

  The proof of (i) and (ii) work as in the proofs of Theorem \ref{Sigma1 wellorder in M1} and Theorem  \ref{Sigma1 wellorder and non CH}. 
  The statement (iii) can be derived from the version of the \emph{condensation lemma} (see {\cite[Theorem 9.3.2]{MR1876087}}) for $M_1$, where the cases (a), (b) and (d) can be ruled out.

 This shows that the tuple $\langle\delta,\omega_2,\omega_1,\lhd\rangle$ is \emph{suitable for $\omega_1$} as in {\cite[Definition 7.1]{MR3591274}}. 
 Suppose that $G$ is $\Add{\omega_1}{\mu}$-generic for some cardinal $\mu<\delta$ with $\cof{\mu}>\omega_1$.
 Then {\cite[Corollary 7.9]{MR3591274}} shows that there is a cofinality preserving forcing extension of  $\VV[G]$ that contains a $\Sigma_1(\omega_1)$-definable well-order of $\HH{\omega_2}$ . 
\end{proof}

The following result shows that the assumption in Theorem \ref{no Delta1-definable Bernstein set} is optimal.

\begin{lemma} \label{construction of Bernstein set} 
 Let $\kappa$ be an uncountable regular cardinal. If there  is a good $\Sigma_1(\kappa)$-definable well-ordering of $\HH{\kappa^+}$, then there is a Bernstein subset of ${}^\kappa\kappa$ that is $\Delta_1(\kappa)$-definable over $\langle\HH{\kappa^+},\in\rangle$. 
\end{lemma} 

\begin{proof} 
A $\Sigma_1(\kappa)$-definable Bernstein set can be constructed by a $\Sigma$-recursion along the good $\Sigma_1(\kappa)$-definable well-ordering $\lhd$ of $\HH{\kappa^+}$. 
We fix a $\Sigma_1(\kappa)$-definable enumeration of perfect subtrees of ${}^{\kappa}\kappa$ of length $\kappa^+$. 
In each step, we choose two distinct elements of the next perfect subset of ${}^{\kappa}\kappa$. We add one of these to the Bernstein set and the other one to its complement. Moreover we add the next element in $\lhd$ either to the Bernstein set or to its complement. 
\end{proof}

\begin{lemma} \label{Delta1 Bernstein set in M_1} 
The existence of a $\Delta_1(\omega_1)$-definable Bernstein subset of ${}^{\omega_1}\omega_1$ is consistent with the existence of a Woodin cardinal. 
\end{lemma} 

\begin{proof} 
This follows from Theorem \ref{Sigma1 wellorder in M1} and Lemma \ref{construction of Bernstein set}.  
\end{proof}


\section{$\Sigma_1(\kappa)$-definable sets at large cardinals} \label{section: large cardinals} 

In this section, we generalize some of the previous results to large cardinals. 

\begin{definition}[\cite{MR2830435},\cite{MR2817562}]\label{definition:IterableCardinal}
 Let $\kappa$ be an uncountable cardinal. 
\begin{enumerate} 
 \item A \emph{weak $\kappa$-model} is a transitive model $M$ of $\mathsf{ZFC}^-$ of size $\kappa$ with $\kappa\in M$. 

 \item The cardinal $\kappa$ is \emph{$\omega_1$-iterable} if for every subset $A$ of $\kappa$ there is a  weak $\kappa$-model $M$ and a weakly amenable $M$-ultrafilter $U$ on $\kappa$ such that $A\in M$ and $\langle M,\in,U\rangle$ is $\omega_1$-iterable. 
\end{enumerate} 
\end{definition}

We start by proving the following analog of Lemma \ref{lemma:HOmega2Sigma13}.

\begin{lemma} \label{characterization of Sigma1(kappa)} 
 Assume that $\kappa$ is either an $\omega_1$-iterable cardinal or a regular cardinal that is a stationary limit of $\omega_1$-iterable cardinals. Then the following statements are equivalent for every subset $X$ of $\RRR$. 
 \begin{enumerate}
  \item The set $X$ is $\Sigma_1(\kappa)$-definable. 

  \item The set $X$ is $\Sigma^1_3$-definable. 
 \end{enumerate} 
\end{lemma} 

\begin{proof} 
 By Lemma \ref{Sigma13HigherCardinals}, it suffices to show that (i) implies (ii). 
 Assume that $\varphi(v_0,v_1)$ is a $\Sigma_1$-formula with $X=\Set{x\in\RRR}{\varphi(\kappa,x)}$. Define $Y$ to be the set of all $y\in\RRR$ with the property that there is a countable transitive model $M$ of $\ZFC^-$, a cardinal $\delta$ of $M$ with $\varphi(\delta,y)^M$ and a weakly amenable $M$-ultrafilter $F$ on $\delta$ such that the structure $\langle M,\in,F\rangle$ is $\omega_1$-iterable.

 \begin{claim*} 
  The set $Y$ is a $\Sigma^1_3$-subset of $\RRR$. 
 \end{claim*} 
 
 \begin{proof} 
    Since $\omega_1$-iterability is a $\Pi^1_2$-statement and all other conditions are first order statements about $\langle M,\in,F\rangle$, the existence of such a structure is a $\Sigma^1_3$-statement. 
 \end{proof} 

 \begin{claim*}
  $X\subseteq Y$. 
 \end{claim*}

 \begin{proof} 
   First, assume that $\kappa$ is $\omega_1$-iterable and pick $x\in X$. Then we can find $A\subseteq\kappa$ with $x\in\LL[A]$ and $\varphi(\kappa,x)^{\LL[A]}$. By our assumption, there is a transitive model $N$ of $\ZFC^-$ of cardinality $\kappa$ with $\kappa,A\in N$ and an $N$-ultrafilter $ U$ on $\kappa$ such that the structure $\langle N,\in, U\rangle$ is iterable. Then $x\in N$ and $\varphi(\kappa,x)^N$.  Let $\langle N_0,\in, U_0\rangle$ be a countable elementary submodel of $\langle N,\in, U\rangle$ with $x,A\in N_0$ and let $\map{\pi}{N_0}{M}$ denote the corresponding transitive collapse. Set $\delta=\pi(\kappa)$ and $F=\pi[ U_0]$. In this situation, {\cite[Theorem 19.15]{MR1994835}} shows that the structure $\langle M,\in,F\rangle$ is iterable. Since $\varphi(\delta,x)^M$ holds by elementarity, we can conclude that $x$ is an element of $Y$.

   Now, assume that $\kappa$ is a stationary limit of $\omega_1$-iterable cardinals. Pick $x\in X$ and a strictly increasing continuous chain $\seq{N_\alpha}{\alpha<\kappa}$ of elementary submodels of $\HH{\kappa^+}$ of cardinality less than $\kappa$ such that $x\in N_0$ and $\kappa_\alpha=\kappa\cap N_\alpha\in\kappa$ for all $\alpha<\kappa$. Then $C=\Set{\kappa_\alpha}{\alpha\in\kappa\cap\Lim}$ is a club in $\kappa$ and there is an $\bar{\kappa}<\kappa$ such that $\kappa_{\bar{\kappa}}$ is $\omega_1$-iterable. 
 Since $\omega_1$-iterability implies inaccessibility, we have $\bar{\kappa}=\kappa_{\bar{\kappa}}$. By elementarity and $\Sigma_1$-upwards absoluteness, we know that $\varphi(\bar{\kappa},x)$ holds. In this situation, we can repeat the construction of the first case to obtain a countable iterable structure $\langle M,\in,F\rangle$ that witnessing that $x$ is an element of $Y$. 
 \end{proof}

  \begin{claim*}
  $Y\subseteq X$. 
 \end{claim*}

 \begin{proof}
   Pick $y\in Y$ and let $\langle M_0,\in,F_0\rangle$ and $\delta\in M_0$ witness this. Then $\langle M_0,\in,F_0\rangle$ is iterable and $\varphi(\delta,y)^{M_0}$ holds. Let $$\langle\seq{\langle M_\alpha,\in,F_\alpha\rangle}{\alpha\in\On},\seq{\map{j_{\bar{\alpha},\alpha}}{M_{\bar{\alpha}}}{M_\alpha}}{\bar{\alpha}\leq\alpha\in\On}\rangle$$ denote the corresponding system of models and elementary embeddings. Then $j_{0,\kappa}(\delta)=\kappa$ and $\varphi(\kappa,y)$ holds by elementarity and $\Sigma_1$-upwards absoluteness. This shows that $y$ is an element of $X$. 
 \end{proof}

  This completes the proof of the lemma. 
 \end{proof}

\begin{lemma}\label{corollary:NoWOLArgeCardinals}
 Assume that $\kappa$ is either an $\omega_1$-iterable cardinal or a regular cardinal that is a stationary limit of $\omega_1$-iterable cardinals. If there is a $\Sigma_1(\kappa)$-definable well-ordering of the reals, then there is a $\Sigma^1_3$-well-ordering of the reals.   \qed 
\end{lemma}

If $\kappa$ is either a Woodin cardinal below a measurable cardinal or a measurable cardinal above a Woodin cardinal, then the above results allow us to show that there is no $\Sigma_1(\kappa)$-definable well-ordering of the reals.

\begin{proof}[Proof of Theorem \ref{theorem:LargeCardinalsWO}]
 Let $\kappa$ either be a measurable cardinal above a Woodin cardinal or a Woodin cardinal below a measurable cardinal. Then $\mathbf{\Sigma}^1_2$-determinacy holds and no well-ordering of the reals is $\Sigma^1_3$-definable. If $\kappa$ is a measurable cardinal, then $\kappa$ is $\omega_1$-iterable (see \cite{MR2830415}) and Corollary \ref{corollary:NoWOLArgeCardinals} implies that no well-ordering of the reals is $\Sigma_1(\kappa)$-definable. In the other case, if $\kappa$ is a Woodin cardinal, then $\kappa$ is a stationary limit of measurable cardinals (and hence a stationary limit of $\omega_1$-iterable cardinals) and Corollary \ref{corollary:NoWOLArgeCardinals} implies that no well-ordering of the reals is $\Sigma_1(\kappa)$-definable. 
\end{proof}

In the following, we prove a large cardinal version of Lemma \ref{lemma:Sigma1SetContainsBistationary}. This result will allow us to prove Theorem \ref{theorem:LargeCardinalsBernstein}.

\begin{lemma}\label{lemma:MeasurableCardinalPerfectSeubset}
 Let $\kappa$ be a measurable cardinal and let $X$ be a $\Sigma_1(\kappa)$-definable subset of ${}^\kappa\kappa$. If there is an $x\in X$ such that are normal ultrafilters $ U_0$ and $ U_1$ on $\kappa$ with $\bar{x}=\Set{\alpha<\kappa}{x(\alpha)=0}\in U_1\setminus U_0$, then for every $\xi<\kappa$ there is 
 \begin{enumerate} 
  \item a continuous injection $\map{\iota}{{}^{\omega_1}2}{X}$ with $\ran{\iota}\subseteq N_{x\restriction\xi}\cap X$ 

  \item a club $D$ in $\kappa$ with monotone enumeration $\seq{\delta_\alpha}{\alpha<\kappa}$
 \end{enumerate} 
such that for all $z\in{}^{\kappa}2$ and $\alpha<\kappa$, we have $z(\alpha)=1$ if and only if $\iota(z)(\delta_\alpha)>0$.
\end{lemma} 

\begin{proof} 
 Fix $\xi<\kappa$ and a regular cardinal $\theta>\kappa$ with $\POT{\POT{\kappa}}\in\HH{\theta}$. Pick a $\Sigma_1$-formula $\varphi(v_0,v_1)$ with  $X=\Set{z\in{}^\kappa\kappa}{\varphi(\kappa,z)}$ and an elementary submodel $N$ of $\HH{\theta}$ of cardinality less than $\kappa$ with $\kappa, x,  U_0,  U_1 \in N$ and $\xi+1\subseteq N$. Let $\map{\pi}{N}{M}$ denote the corresponding transitive collapse. 

In this situation {\cite[Theorem 2.3]{SteelLectureNotes}} shows that there is a directed system $$\langle\seq{M_s}{s\in{}^{{\leq}\kappa}2}, ~ \seq{\map{j_{s,t}}{M_s}{M_t}}{s,t\in{}^{{\leq}\kappa}2, ~ s\subseteq t}\rangle$$ of transitive models of $\ZFC^-$ and elementary embeddings such that the following statements hold: 
 \begin{enumerate}
  \item $M=M_\emptyset$. 

  \item If $s\in{}^{{<}\kappa}2$ and $i<2$, then $M_{s^\frown\langle i\rangle}=\Ult{M_s}{(j_{\emptyset,s}\circ\pi)( U_i)}$ and $j_{s,s^\frown\langle i\rangle}$ is the corresponding ultrapower map induced by $(j_{\emptyset,s}\circ\pi)( U_i)$.  

  \item If $s\in{}^{{\leq}\kappa}2$ with $\length{s}\in\Lim$, then $$\langle M_s, ~ \seq{\map{j_{s\restriction\alpha,s}}{M_{s\restriction\alpha}}{M_s}}{\alpha<\length{s}}\rangle$$ is the direct limit of the directed system $$\langle\seq{M_{s\restriction\alpha}}{\alpha<\length{s}}, ~ \seq{\map{j_{s\restriction\bar{\alpha},s\restriction\alpha}}{M_{s\restriction\bar{\alpha}}}{M_{s\restriction\alpha}}}{\bar{\alpha}\leq\alpha<\length{s}}\rangle.$$ 
 \end{enumerate}

 Set $j_s=j_{\emptyset,s}$ for all $s\in{}^{{\leq}\kappa}2$.  
 Since $\kappa=(j_z\circ \pi)(\kappa)$ for all $z\in{}^\kappa2$, we can define $$\Map{i}{{}^\kappa 2}{{}^\kappa\kappa}{z}{(j_z\circ\pi)(x)}.$$

 In this situation, elementarity and $\Sigma_1$-upwards absoluteness imply that $\varphi(\kappa,i(z))$ and $x\restriction\xi=i(z)\restriction\xi$ hold for  all $z\in{}^\kappa 2$. In particular, we have  $\ran{i}\subseteq N_{x\restriction\xi}\cap X$.

 Given $z\in{}^\kappa 2$, we define  $$\Map{c_z}{\kappa}{\kappa}{\alpha}{(j_{z\restriction\alpha}\circ\pi)(\kappa)}.$$ 
Then $\ran{c_z}$ is strictly increasing and continuous for every $z\in{}^\kappa 2$. By definition, we have $c_{z_0}\restriction\alpha=c_{z_1}\restriction\alpha$  for all $z_0,z_1\in{}^\kappa 2$ and $\alpha<\kappa$ with $z_0\restriction\alpha= z_1\restriction\alpha$. Given $z\in{}^\kappa 2$ and $\alpha<\kappa$, we have  $$\crit{j_{z\restriction\alpha,z\restriction(\alpha+1)}} ~ = ~ c_z(\alpha) ~ < ~ c_z(\alpha+1) ~ = ~ \crit{j_{z\restriction(\alpha+1),z}}$$ and $$(j_{z\restriction\alpha}\circ\pi)(\bar{x}) ~ \in ~ (j_{z\restriction\alpha}\circ\pi)( U_1)\setminus(j_{z\restriction\alpha}\circ\pi)(U_0).$$ This allows us to conclude that   
 \begin{equation*}
  \begin{split}
   z(\alpha)=1 ~ &  \Longleftrightarrow ~ c_z(\alpha)\in (j_{z\restriction(\alpha+1)}\circ\pi)(\bar{x})  \\ 
   & \Longleftrightarrow ~ (
((j_{z\restriction(\alpha+1)}\circ\pi)(x))(c_z(\alpha))>0 \\ 
   & \Longleftrightarrow ~ (
((j_{z\restriction(\alpha+1),z}\circ j_{z\restriction(\alpha+1)}\circ\pi)(x))(c_z(\alpha))>0 \\ 
  &  \Longleftrightarrow ~ (i(z) (c_z(\alpha))>0
  \end{split}
 \end{equation*}
holds for all $z\in{}^\kappa 2$ and $\alpha<\kappa$. In particular, this shows that  $i$ is injective. 

  Now, fix $z\in{}^{\omega_1}2$ and $\beta<\omega_1$. Pick $\alpha<\omega_1$  with $c_z(\alpha)>\beta$. Since we have $c_{\bar{z}}(\alpha)=\crit{j_{\bar{z}\restriction\alpha,z}}$ and $i(\bar{z})\restriction\beta = (j_{\bar{z}\restriction\alpha}\circ\pi)(x)\restriction\beta$ for all $\bar{z}\in{}^{\omega_1}2$, we can conclude that $i(z)\restriction\beta=i(\bar{z})\restriction\beta$ holds for all $\bar{z}\in N_{z\restriction\alpha}\cap{}^\kappa 2$.   This shows that $i$ is  continuous.

  Let $\seq{\delta_\alpha}{\alpha<\kappa}$ denote the monotone enumeration of the club $D$ of all uncountable cardinals less than $\kappa$ and let $\map{e}{{}^\kappa 2}{{}^\kappa 2}$ denote the unique continuous injection with $e(z)^{{-}1}\{1\}=\Set{\delta_\alpha}{\alpha<\kappa, ~ z(\alpha)=1}$ for all $z\in{}^\kappa 2$. Then $c_z\restriction D=\id_D$ for all $z\in {}^\kappa 2$. Set $\iota=i\circ e$. Given $z\in{}^\kappa 2$ and $\alpha<\kappa$, we then have 
 \begin{equation*}
  z(\alpha)=1 ~ \Longleftrightarrow ~ e(z)(\delta_\alpha)=1 ~ \Longleftrightarrow ~ i(e(z)) (c_{e(z)}(\delta_\alpha))>0 ~ \Longleftrightarrow ~ \iota(z)(\delta_\alpha)>0. \qedhere
 \end{equation*}
\end{proof}

The above lemma allows us to prove the following strengthening of Theorem \ref{theorem:LargeCardinalsBernstein}.

\begin{theorem}
 Let  $\kappa$ be a measurable cardinal with the property that there are two distinct  normal ultrafilters on $\kappa$ and let $\Gamma$ be a set of $\Sigma_1(\kappa)$-definable subsets of ${}^\kappa\kappa$. If $\bigcup\Gamma={}^\kappa\kappa$, then some element of $\Gamma$ contains a perfect subset. 
\end{theorem}

\begin{proof}
 Pick normal ultrafilters $ U_0$ and $ U_1$ on $\kappa$ with $ U_0\neq U_1$. Then there is $x\in{}^\kappa\kappa$ with $\Set{\alpha<\kappa}{x(\alpha)>0}\in U_1\setminus U_0$ and  $X\in\Gamma$ with $x\in X$.  In this situation, Lemma \ref{lemma:MeasurableCardinalPerfectSeubset} implies that $X$ contains a perfect subset. 
\end{proof}

The following result shows that the conclusion of Theorem \ref{theorem:LargeCardinalsBernstein} does not hold for all measurable cardinals.

\begin{theorem} 
 Assume that $\delta$ is a measurable cardinal and $U$ is a normal ultrafilter on $\delta$ with $\VV=\LL[U]$. If $\kappa\leq\delta$ is an uncountable regular cardinal, then there is a Bernstein subset of ${}^\kappa\kappa$ that is $\Delta_1(\kappa)$-definable over $\langle\HH{\kappa^+},\in\rangle$. 
\end{theorem} 

\begin{proof} 
  Following {\cite[p. 264]{MR1994835}}, we define a \emph{$\ZFC^-$-mouse at $\lambda$} to be a structure $\langle M,\in,F\rangle$ such that $M$ is a transitive model of $\ZFC^-$ with $M=\LL_\alpha[F]$ for some ordinal $\alpha$ and $F$ is a weakly amenable $M$-ultrafilter on $\lambda$ such that $\langle M,\in,F\rangle$ is $\omega_1$-iterable.  Note that $\omega_1$-iterability implies full iterability and $\HH{\delta^+}\subseteq\Ult{\VV}{U}$ implies that every element of $\HH{\delta^+}$ is contained in a $\ZFC^-$-mouse at some $\lambda>\delta$. 

 Given an uncountable regular cardinal $\kappa\leq\delta$, we define a well-order $\lhd$ on $\HH{\kappa^+}$ by setting  $x\lhd y$ if there is a $\ZFC^-$-mouse $\langle M,\in,F\rangle$ at some $\lambda>\kappa$ with $x,y\in M$ and $x<_{\LL[F]}y$.

\begin{claim*} 
$\lhd$ is a good $\Sigma_1(\kappa)$-definable well-order of $\POT{\kappa}^{\LL[U]}$. 
\end{claim*} 

\begin{proof} 
 Let $M$ be a $\ZFC^-$-mouse.  
 By {\cite[Lemma 20.8]{MR1994835}}, there are elementary embeddings $\map{i}{M}{\LL_\gamma[F]}$ and  $\map{j}{\Ult{\VV}{U}}{\LL[F]}$ with critical points greater than $\kappa$ and $\POT{\kappa}^M =\POT{\kappa}^{\LL_\gamma[F]}\subseteq\POT{\kappa}^{\LL[F]} =\POT{\kappa}^\VV$. Hence $\lhd$ is equal to the restriction of the canonical well-order of $\Ult{\VV}{U}$ to $\HH{\kappa^+}^\VV$ and every $\ZFC^-$-mouse is downwards-closed with respect to $\lhd$. Since $\omega_1$-iterability can be checked by transitive models of some fragments of $\ZFC$ containing $\omega_1$ as a subset and is therefore a $\Sigma_1(\kappa)$ condition, the above computations yield the statement of the claim.  
\end{proof}

By Lemma \ref{construction of Bernstein set}, the above claim implies the statement of the theorem. 
\end{proof}

In the remainder of this section, we study the $\Pi_1$-definability of the club filter at large cardinals. We start by proving Theorem \ref{theorem:IterableCardinalsClubFilter}, which shows that the club filter on $\kappa$ is not $\Pi_1(\kappa)$-definable if $\kappa$ is a stationary limit of $\omega_1$-iterable cardinals. 

\begin{proof}[Proof of Theorem \ref{theorem:IterableCardinalsClubFilter}]
 Let $\kappa$ be a regular cardinal that is a stationary limit of $\omega_1$-iterable cardinals. 
  Fix a $\Sigma_1$-formula $\varphi(v_0,v_1)$ and assume, towards a contradiction, that the complement of the club filter on $\kappa$ is equal to the set $\Set{x\subseteq\kappa}{\varphi(\kappa,x)}$. Let $y$ denote the set of $\omega_1$-iterable cardinals less than $\kappa$ and set $z=\kappa\setminus y$. Then $z$ is a bistationary subset of $\kappa$ and $\varphi(\kappa,z)$ holds. 

Pick a strictly increasing continuous chain $\seq{N_\alpha}{\alpha<\kappa}$ of elementary submodels of $\HH{\kappa^+}$ of cardinality less than $\kappa$ such that $z\in N_0$ and $\kappa_\alpha=\kappa\cap N_\alpha\in\kappa$ for all $\alpha<\kappa$. Then $C=\Set{\kappa_\alpha}{\alpha\in\kappa\cap\Lim}$ is a club in $\kappa$. Let $\delta$ denote the minimal element of $\kappa\cap\Lim$ with $\kappa_\delta\in y$. Since $\kappa_\delta$ is an $\omega_1$-iterable cardinal and therefore regular, we know that $\delta=\kappa_\delta$. Let $\map{\pi}{N_\delta}{N}$ denote the transitive collapse of $N_\delta$. Then $\pi(\kappa)=\delta$ and  $\pi(z)=z\cap\delta$. In this situation,  $\Sigma_1$-upwards absoluteness implies that $\varphi(\delta,z\cap\delta)$ holds in $\VV$. Moreover, $C\cap\delta$ is a club in $\delta$ and the minimality of $\delta$ implies that $C\cap\delta$ is a subset of $z\cap\delta$.

Since $\delta$ is $\omega_1$-iterable, we can find a weak $\delta$-model $M_0$ and an $M_0$-ultrafilter $F_0$ on $\delta$ such that $z\cap\delta,C\cap\delta\in M_0$, $\varphi(\delta,z\cap\delta)^{M_0}$ holds and $\langle M_0,\in,F_0\rangle$ is iterable. Let $$\langle\seq{\langle M_\alpha,\in,F_\alpha\rangle}{\alpha\in\On},\seq{\map{j_{\bar{\alpha},\alpha}}{M_{\bar{\alpha}}}{M_\alpha}}{\bar{\alpha}\leq\alpha\in\On}\rangle$$ denote the corresponding system of models and elementary embeddings. Then $j_{0,\kappa}(\delta)=\kappa$  and $j_{0,\kappa}(C\cap\delta)$ is a club in $\kappa$ that witnesses that the set $j_{0,\kappa}(z\cap\delta)$ is contained in the club filter on $\kappa$. But $\Sigma_1$-upwards absoluteness and elementarity imply that $\varphi(\kappa,j_{0,\kappa}(z\cap\delta))$ holds, a contradiction. 
\end{proof}

Next, we prove an analog of Lemma \ref{lemma:Omega1DenseClubNS} for certain large cardinals.

\begin{lemma}\label{lemma:NormalUltrafilterNoSeparation}
 Let $\kappa$ be an uncountable regular cardinal, let $M$ be a weak $\kappa$-model and let $U$ be an $M$-ultrafilter such that $\langle M,\in,  U\rangle$ is $\omega_1$-iterable. If $\varphi(v_0,v_1)$ is a $\Sigma_1$-formula, then the following statements hold for all $\xi<\kappa$ and $x\in M\cap\POT{\kappa}$ with the property that $\varphi(\kappa,x)^M$ holds: 
 \begin{enumerate}
  \item If $x\in U$, then there is an element $y$ of the club filter on $\kappa$ such that $x\restriction\xi=y\restriction\xi$ and $\varphi(\kappa,y)$ holds.  

  \item If $x\notin U$, then there is an element $y$ of the nonstationary ideal on $\kappa$ such that $x\restriction\xi=y\restriction\xi$ and $\varphi(\kappa,y)$ holds.  
 \end{enumerate}
\end{lemma}

\begin{proof}
 Pick an elementary submodel $\langle N,\in, F\rangle$ of $\langle M,\in, U\rangle$ of cardinality less than $\kappa$ with $\kappa,x\in N$ and $\xi+1\subseteq N$. Let $\map{\pi}{N}{M_0}$ denote the corresponding transitive collapse. Set $ F_0=\pi[F]$. Then $F_0$ is an $M_0$-ultrafilter and {\cite[Theorem 19.15]{MR1994835}} implies that the structure $\langle M_0,\in,F_0\rangle$ is iterable. Let $$\langle\seq{\langle M_\alpha,\in,F_\alpha\rangle}{\alpha\in\On},\seq{\map{j_{\bar{\alpha},\alpha}}{M_{\bar{\alpha}}}{M_\alpha}}{\bar{\alpha}\leq\alpha\in\On}\rangle$$ denote the corresponding system of models and elementary embeddings. Define $y=(j_{0,\kappa}\circ\pi)(x)$. Since $\kappa=(j_{0,\kappa}\circ\pi)(\kappa)$, $\Sigma_1$-upwards absoluteness and elementarity imply that $\varphi(\kappa,y)$ holds and $x\restriction\xi=y\restriction\xi$. Moreover, the set $C=\Set{(j_{0,\alpha}\circ\pi)(\kappa)}{\alpha<\kappa}$ is a club in $\kappa$.

 Now, assume $x\in U$. Then $(j_{0,\alpha}\circ\pi)(x)\in  F_\alpha$ and $(j_{0,\alpha}\circ\pi)(\kappa)\in(j_{0,\alpha+1}\circ\pi)(x)$ for all $\alpha<\kappa$. Since we have $(j_{0,\alpha}\circ\pi)(x)<(j_{0,\alpha+1}\circ\pi)(x)=\crit{j_{\alpha+1,\kappa}}$ for all $\alpha<\kappa$, we can conclude that $C$ is a subset of $y$ in this case and therefore $y$ is contained in the club filter on $\kappa$. 

 Finally, assume $x\notin U$. Then $(j_{0,\alpha}\circ\pi)(x)\notin F_\alpha$ and $(j_{0,\alpha}\circ\pi)(\kappa)\notin(j_{0,\alpha+1}\circ\pi)(x)$ for all $\alpha<\kappa$. As above, we can conclude that $C$ is disjoint from $y$ in this case and therefore $y$ is an element of the nonstationary ideal.  
\end{proof}

The previous lemma allows us to show that the club filter and the non-stationary ideal cannot be separated by a $\Delta_1(\kappa)$-set for certain large cardinals $\kappa$.

\begin{proof}[Proof of Theorem \ref{theorem:MeasurableCardinalClubFilter}]
 Let $\kappa$ be an $\omega_1$-iterable cardinal and assume that there are $\Sigma_1$-formulas  $\varphi(v_0,v_1)$ and $\psi(v_0,v_1)$ with the property that the subset $X=\Set{x\subseteq\kappa}{\varphi(\kappa,x)}$ of $\POT{\kappa}$ separates the club filter from the nonstationary ideal and $\POT{\kappa}\setminus X=\Set{x\subseteq\kappa}{\psi(\kappa,x)}$. Pick an elementary submodel $M$ of $\HH{\kappa^+}$ of cardinality $\kappa$ with $\kappa+1\subseteq M$.  
 By our assumptions, there is a $\kappa$-model $N$ and an $N$-ultrafilter $ U$ on $\kappa$ such that $M\in N$ and $\langle N,\in, U\rangle$ is iterable. Set $F=M\cap U$.

 \begin{claim*}
 $F=M\cap X$.
\end{claim*} 

\begin{proof} 
 Assume that there is $x\in F$ with $x\notin X$. Then elementarity implies that $\psi(\kappa,x)^M$ holds and $\Sigma_1$-upwards absoluteness implies that $\psi(\kappa,x)^N$ holds. By the first part of Lemma \ref{lemma:NormalUltrafilterNoSeparation}, this shows that there is an element $y$ of the club filter on $\kappa$ such that $\psi(\kappa,y)$ holds, a contradiction. This shows that $ F\subseteq M\cap X$.

 Now, assume that $x\in M\cap X$ with $x\notin U$. Then elementarity implies that $\varphi(\kappa,x)^M$ holds and $\Sigma_1$-upwards absoluteness implies that $\varphi(\kappa,x)^N$ holds. By the second part of Lemma \ref{lemma:NormalUltrafilterNoSeparation}, there is an element $y$ of the nonstationary ideal on $\kappa$ such that $\varphi(\kappa,y)$ holds, a contradiction. Together with the above computations, this shows that $F=M\cap X$.
 \end{proof} 
 
 Since $\langle M,\in,F\rangle\models\anf{\textit{$F$ is a normal ultrafilter on $\kappa$}}$ and $F$ is $\Delta_1(\kappa)$-definable over $\langle M,\in\rangle$, elementarity implies that $X$ is a normal ultrafilter over $\kappa$ in $\VV$. Let $\Ult{\VV}{X}$ denote the corresponding ultrapower of $\VV$. Then $\HH{\kappa^+}=\HH{\kappa^+}^{\Ult{\VV}{X}}$. Since $X$ is definable over $\langle\HH{\kappa^+},\in\rangle$, we can conclude that $X$ is an element of $\Ult{\VV}{X}$, a contradiction.  
\end{proof}

For measurable cardinals $\kappa$, we obtain a  result similar to Lemma \ref{Sigma1 absoluteness}.

\begin{lemma} \label{Sigma1(kappa)-absoluteness} 
 Let $\kappa$ be an $\omega_1$-iterable cardinal and let $\lambda$ be a measurable cardinal. 
\begin{enumerate} 
 \item $\Sigma_1(\kappa)$-statements (without parameters) are absolute to generic extensions for forcings of size less than $\lambda$ which preserve the $\omega_1$-iterability of $\kappa$.  

 \item  The set of $\Sigma_1(\kappa)$-definable singletons $\{x\}$ with $x\subseteq\kappa$ is absolute for forcings of size less than $\lambda$ which preserve the $\omega_1$-iterability of $\kappa$. 
 \end{enumerate} 
\end{lemma} 

\begin{proof} 
The first claim follows from Lemma \ref{characterization of Sigma1(kappa)}, since the statement is equivalent to a $\Sigma^1_3$-statement and $\Sigma^1_3$-absoluteness holds for forcings of size less than $\lambda$ (see {\cite[Lemma 3.7]{MR3226056}}). The second claim follows from the first claim. 
\end{proof}

Note that, if $\kappa$ is an $\omega_1$-iterable cardinal, then forcing with a partial order of cardinality less than $\kappa$ preserves the $\omega_1$-iterability of $\kappa$.


\section{Open questions}\label{section:Questions}

We close this paper with a collection of questions raised by the above results.

 First, Lemma \ref{lemma:HOmega2Sigma13} and Lemma \ref{characterization of Sigma1 with uB parameter} suggest the following question. 

\begin{question} 
Assume that there is a proper class of Woodin cardinals. If $B$ is a uB set of reals, is every $\Sigma^1_3(B)$-set $\Sigma_1(\omega_1)$-definable over $\langle\HH{\omega_2},\in,B,\mathrm{NS}_{\omega_1}\rangle$? 
\end{question} 

Theorem \ref{Sigma1-definable club uniformization} leaves open the following question. 

\begin{question} 
Suppose that there is a Woodin cardinal and a measurable cardinal above it. Is there no $\Sigma_1(\omega_1)$-definable uniformization of the club filter on $\omega_1$? 
\end{question}

Note that the existence of a good $\Sigma_1(\omega_1)$-definable well-order of $\POT{\omega_1}$ yields a $\Sigma_1(\omega_1)$-definable uniformization of the club filter on $\omega_1$ and  Theorem \ref{Sigma1 wellorder in M1} shows that  such a uniformization is compatible with the existence of a Woodin cardinal.

Next, we ask if the assumption in Theorem \ref{perfect subsets} is optimal. The conclusion does not follow from the existence of a Woodin cardinal by  the proof of Lemma \ref{Delta1 Bernstein set in M_1}. Moreover, the perfect set property for all definable subsets of ${}^{\omega_1}\omega_1$ can be forced by Levy-collapsing an inaccessible cardinal (see \cite{PS}).

\begin{question} 
Suppose that $\mathrm{NS}_{\omega_1}$ is saturated or that there is a Woodin cardinal and a measurable cardinal above it. 
Does the perfect set dichotomy over $\LL(\RRR)$ in Theorem \ref{perfect subsets} hold? 
\end{question}

Moreover, we do not know if the two cases in the perfect set dichotomy in Theorem \ref{perfect subsets} are mutually exclusive unless $2^{\omega}<2^{\omega_1}$. 
This is related to the question over which models it is possible to add perfect subsets of the ground model (see {\cite[Lemma 6.2]{MR3557473}} and {\cite{MR1640916}}).

\begin{question} 
Is it consistent with the existence of a Woodin cardinal and a measurable cardinal above it that there is a perfect subset of ${}^{\omega_1}\omega_1\cap\LL(\RRR)$? In particular, does this statement fail in the $\mathbb{P}_{max}$-extension of $\LL(\RRR)$ if there are infinitely many Woodin cardinals? 
\end{question} 

We ask about generalizations of the results of this paper to $\omega_2$ and larger cardinals. 
In this situation, the method of iterations of generic ultrapowers fails, since generics need not exist over uncountable models. 

\begin{question} 
 Is the existence of a $\Sigma_1(\omega_2)$-definable well-ordering of the reals  compatible with the existence of a supercompact cardinal? 
\end{question}

We also ask about a perfect set dichotomy for large cardinals. 

\begin{question}
 Let $\kappa$ be a supercompact cardinal and let $X$ be a subset of ${}^\kappa\kappa$ that is $\Sigma_1$-definable over $\langle\HH{\kappa^+},\in\rangle$ and has cardinality greater than $\kappa$. Does $X$ contain a perfect subset? 
\end{question}

The motivation for this question is that for supercompact cardinals, there are many different normal ultrafilters on $\kappa$. 
Let $\kappa$ be a measurable cardinal and let $D$ denote the collection of all subsets $y$ of $\kappa$ with the property that there are ultrapowers $I_0$ and $I_1$ of $\VV$ with normal ultrafilters on $\kappa$ such that 
$j_{I_0}(\kappa)=j_{I_1}(\kappa)$ and $j_{I_0}(y)\neq j_{I_1}(y)$. 
Then 
 the above proofs show: If $X$ is a $\Sigma_1(\kappa)$-definable subset of ${}^\kappa\kappa$ and there is an element $x$ of $X$ with $\Set{\alpha<\kappa}{x(\alpha)>0}\in D$, then $X$ contains a perfect subset.

 Finally, Lemma \ref{Sigma1(kappa)-absoluteness} leaves open the following question. 
 
\begin{question} 
Suppose that $\Phi(\kappa)$ holds, where $\Phi(\kappa)$ is a large cardinal property that implies that $\kappa$ is weakly compact. 
Are $\Sigma_1(\kappa)$-formulas with parameters in $\HH{\kappa^+}$ absolute to generic extensions for ${<}\kappa$-distributive forcings which preserve $\Phi(\kappa)$?  
\end{question}

\bibliographystyle{plain}
 \bibliography{references}

\begin{thebibliography}{10}

\bibitem{MR2231126}
Andr{\'e}s~Eduardo Caicedo and Boban Veli{\v{c}}kovi{\'c}.
\newblock The bounded proper forcing axiom and well orderings of the reals.
\newblock {\em Math. Res. Lett.}, 13(2-3):393--408, 2006.

\bibitem{MR2963017}
Benjamin Claverie and Ralf Schindler.
\newblock Woodin's axiom {$(\ast)$}, bounded forcing axioms, and precipitous
  ideals on {$\omega_1$}.
\newblock {\em J. Symbolic Logic}, 77(2):475--498, 2012.

\bibitem{MR2374762}
Ilijas Farah and Paul~B. Larson.
\newblock Absoluteness for universally {B}aire sets and the uncountable. {I}.
\newblock In {\em Set theory: recent trends and applications}, volume~17 of
  {\em Quad. Mat.}, pages 47--92. Dept. Math., Seconda Univ. Napoli, Caserta,
  2006.

\bibitem{FriedmanWu}
Sy-David Friedman and Liuzhen Wu.
\newblock Large cardinals and {$\Delta_1$}-definablity of the nonstationary
  ideal.
\newblock Preprint.

\bibitem{MR3320477}
Sy-David Friedman, Liuzhen Wu, and Lyubomyr Zdomskyy.
\newblock {$\Delta_1$}-definability of the non-stationary ideal at successor
  cardinals.
\newblock {\em Fund. Math.}, 229(3):231--254, 2015.

\bibitem{MR2830415}
Victoria Gitman.
\newblock Ramsey-like cardinals.
\newblock {\em J. Symbolic Logic}, 76(2):519--540, 2011.

\bibitem{MR2830435}
Victoria Gitman and Philip~D. Welch.
\newblock Ramsey-like cardinals {II}.
\newblock {\em J. Symbolic Logic}, 76(2):541--560, 2011.

\bibitem{MR1324501}
Martin Goldstern and Saharon Shelah.
\newblock The bounded proper forcing axiom.
\newblock {\em J. Symbolic Logic}, 60(1):58--73, 1995.

\bibitem{HL}
Peter Holy and Philipp L{\"u}cke.
\newblock Locally {$\Sigma_1$}-definable well-orders of
  {${\mathrm{H}}(\kappa^+)$}.
\newblock {\em Fund. Math.}, 226(3):221--236, 2014.

\bibitem{MR3591274}
Peter Holy and Philipp L{\"u}cke.
\newblock Simplest possible locally definable well-orders.
\newblock {\em Fund. Math.}, 236(2):101--139, 2017.

\bibitem{MR1940513}
Thomas Jech.
\newblock {\em Set theory}.
\newblock Springer Monographs in Mathematics. Springer-Verlag, Berlin, 2003.
\newblock The third millennium edition, revised and expanded.

\bibitem{MR560220}
Thomas Jech, Menachem Magidor, William~J. Mitchell, and Karel Prikry.
\newblock Precipitous ideals.
\newblock {\em J. Symbolic Logic}, 45(1):1--8, 1980.

\bibitem{MR3135495}
Ronald Jensen and John~R. Steel.
\newblock {$K$} without the measurable.
\newblock {\em J. Symbolic Logic}, 78(3):708--734, 2013.

\bibitem{MR1994835}
Akihiro Kanamori.
\newblock {\em The higher infinite}.
\newblock Springer Monographs in Mathematics. Springer-Verlag, Berlin, second
  edition, 2003.
\newblock Large cardinals in set theory from their beginnings.

\bibitem{MR2069032}
Paul~B. Larson.
\newblock {\em The stationary tower}, volume~32 of {\em University Lecture
  Series}.
\newblock American Mathematical Society, Providence, RI, 2004.
\newblock Notes on a course by W. Hugh Woodin.

\bibitem{MR2768703}
Paul~B. Larson.
\newblock Forcing over models of determinacy.
\newblock In {\em Handbook of set theory. {V}ols. 1, 2, 3}, pages 2121--2177.
  Springer, Dordrecht, 2010.

\bibitem{LarsonBriefHistoryOfDeterminacy}
Paul~B. Larson.
\newblock A brief history of determinacy.
\newblock In {\em The Handbook of the History of Logic, volume 6}. Elsevier,
  2012.

\bibitem{MR2987148}
Philipp L{\"u}cke.
\newblock {$\Sigma^1_1$}-definability at uncountable regular cardinals.
\newblock {\em J. Symbolic Logic}, 77(3):1011--1046, 2012.

\bibitem{MR3557473}
Philipp L{\"u}cke, Luca Motto~Ros, and Philipp Schlicht.
\newblock The {H}urewicz dichotomy for generalized {B}aire spaces.
\newblock {\em Israel J. Math.}, 216(2):973--1022, 2016.

\bibitem{MR576464}
Menachem Magidor.
\newblock Precipitous ideals and {${\bf \Sigma }_{4}^{1}$} sets.
\newblock {\em Israel J. Math.}, 35(1-2):109--134, 1980.

\bibitem{MR1242054}
Alan Mekler and Jouko V{\"a}{\"a}n{\"a}nen.
\newblock Trees and {$\Pi^1_1$}-subsets of {${}^{\omega_1}\omega_1$}.
\newblock {\em J. Symbolic Logic}, 58(3):1052--1070, 1993.

\bibitem{MR1300637}
William~J. Mitchell and John~R. Steel.
\newblock {\em Fine structure and iteration trees}, volume~3 of {\em Lecture
  Notes in Logic}.
\newblock Springer-Verlag, Berlin, 1994.

\bibitem{MR2526093}
Yiannis~N. Moschovakis.
\newblock {\em Descriptive set theory}, volume 155 of {\em Mathematical Surveys
  and Monographs}.
\newblock American Mathematical Society, Providence, RI, second edition, 2009.

\bibitem{MR1349683}
Itay Neeman.
\newblock Optimal proofs of determinacy.
\newblock {\em Bull. Symbolic Logic}, 1(3):327--339, 1995.

\bibitem{MR2768699}
Ernest Schimmerling.
\newblock A core model toolbox and guide.
\newblock In {\em Handbook of set theory. {V}ols. 1, 2, 3}, pages 1685--1751.
  Springer, Dordrecht, 2010.

\bibitem{MR2096166}
Ralf Schindler.
\newblock Semi-proper forcing, remarkable cardinals, and bounded {M}artin's
  maximum.
\newblock {\em MLQ Math. Log. Q.}, 50(6):527--532, 2004.

\bibitem{PS}
Philipp Schlicht.
\newblock Perfect subsets of generalized {B}aire spaces and long games.
\newblock Submitted.

\bibitem{MR3226056}
Philipp Schlicht.
\newblock Thin equivalence relations and inner models.
\newblock {\em Ann. Pure Appl. Logic}, 165(10):1577--1625, 2014.

\bibitem{MR2817562}
Ian Sharpe and Philip~D. Welch.
\newblock Greatly {E}rd{\H o}s cardinals with some generalizations to the
  {C}hang and {R}amsey properties.
\newblock {\em Ann. Pure Appl. Logic}, 162(11):863--902, 2011.

\bibitem{SteelLectureNotes}
John~R. Steel.
\newblock Introduction to iterated ultrapowers.
\newblock Lecture notes.

\bibitem{MR1257469}
John~R. Steel.
\newblock Inner models with many {W}oodin cardinals.
\newblock {\em Ann. Pure Appl. Logic}, 65(2):185--209, 1993.

\bibitem{MR2768698}
John~R. Steel.
\newblock An outline of inner model theory.
\newblock In {\em Handbook of set theory. {V}ols. 1, 2, 3}, pages 1595--1684.
  Springer, Dordrecht, 2010.

\bibitem{CabalVolume3}
John~R. Steel and W.~Hugh Woodin.
\newblock {HOD} as a core model.
\newblock In {\em Ordinal definability and recursion theory. {T}he {C}abal
  {S}eminar. {V}olume {III}}, volume~43 of {\em Lecture Notes in Logic}, pages
  257--348. Association for Symbolic Logic, La Jolla, CA; Cambridge University
  Press, Cambridge, 2016.

\bibitem{MR1640916}
Boban Veli{\v{c}}kovi{\'c} and W.~Hugh Woodin.
\newblock Complexity of reals in inner models of set theory.
\newblock {\em Ann. Pure Appl. Logic}, 92(3):283--295, 1998.

\bibitem{MR1713438}
W.~Hugh Woodin.
\newblock {\em The axiom of determinacy, forcing axioms, and the nonstationary
  ideal}, volume~1 of {\em de Gruyter Series in Logic and its Applications}.
\newblock Walter de Gruyter \& Co., Berlin, 1999.

\bibitem{MR1876087}
Martin Zeman.
\newblock {\em Inner models and large cardinals}, volume~5 of {\em de Gruyter
  Series in Logic and its Applications}.
\newblock Walter de Gruyter \& Co., Berlin, 2002.

\end{thebibliography}

\end{document}